\documentclass[12pt,a4paper]{article}
\usepackage[utf8]{inputenc}
\usepackage[T1]{fontenc}
\usepackage[english]{babel}
\usepackage{amsmath, amssymb, amsthm}
\usepackage{mathrsfs}
\usepackage{hyperref}
\usepackage{geometry}
\usepackage{tikz}
\usepackage{pgfplots}
\usepackage{enumitem}
\newcommand{\Renyi}{\mathcal{R}_\varepsilon}
\newcommand{\Wp}{\mathcal{W}_{p(\cdot)}}
\newcommand{\Wb}{\mathcal{W}_2}
\newcommand{\Wt}{\mathcal{W}_2}
\newcommand{\vol}{\operatorname{vol}}
\newcommand{\Ric}{\operatorname{Ric}}
\newcommand{\Hess}{\operatorname{Hess}}
\newcommand{\supp}{\operatorname{supp}}
\newcommand{\Ent}{\operatorname{Ent}}
\newcommand{\I}{\mathcal{I}}
\newcommand{\grad}{\nabla}
\newcommand{\divergence}{\nabla\cdot}

\newcommand{\CD}{\operatorname{CD}}

\newtheorem{theorem}{Theorem}[section]
\newtheorem{lemma}[theorem]{Lemma}
\newtheorem{proposition}[theorem]{Proposition}
\newtheorem{corollary}[theorem]{Corollary}
\newtheorem{definition}[theorem]{Definition}
\newtheorem{remark}[theorem]{Remark}
\newtheorem{assumption}[theorem]{Assumption}

\title{Variable Exponent Wasserstein Spaces: Stability of Entropy Convexity and Modified R\'enyi Entropy}
\author{Ambroise Soglo$^1$, Cyrille Comb\'et\'e$^1$, Koffi W. Hou\'edanou$^2$ and  Léonard Todjihound\'e$^1$
	\\
	$^1$ Institut de Mathématiques et de Sciences Physiques (IMSP)\\
	$^2$ D\'epartement de Mathématiques \\ \small{Université d'Abomey-Calavi (UAC), Bénin}
}

\date{\today}

\begin{document}
	
	\maketitle
	
	\begin{abstract}
		We study the Wasserstein space $\mathcal{P}(M)$ equipped with a distance $\Wp$ constructed from the Lagrangian $L(x,v)=|v|^{p(x)}$ where $p(x)=2+\varepsilon(x)$ with $\varepsilon$ small. Building on the fundamental work of Lott and Villani on the $K$-geodesic convexity of the Boltzmann entropy in $(\mathcal{P}(M),\Wb)$, we establish a generalized inequality showing that the entropy remains $\left(K - C\|\varepsilon\|_\infty\right)$-convex along $\Wp$-geodesics. We then introduce a modified R\'enyi entropy that exactly compensates the logarithmic divergence that appears in the expansions of $\Wp^2$, obtaining thus a sharp equivalence that reaveals the Bakry-\'Emery tensor as the effective curvature in the variable exponent setting. As applications, we derive perturbed versions of the Log-Sobolev and Talagrand inequalities in variable exponent Wasserstein spaces, showing that these fundamental functional inequalities are robust under small perturbations of the transport exponent. This work generalizes the Lott-Villani theorem and its consequences (\emph{J. Lott and C. Villani, Ann. of Math. \textbf{169} (2009), 903-991}) to situations where the transport metric varies spatially.
	\end{abstract}
	\footnote{\emph{The authors gratefully acknowledge the Institut de Mathématiques et de Sciences
		Physiques (IMSP), Université d'Abomey-Calavi, Bénin, for its continuous
		institutional and scientific support. The research environment and resources
		provided by the IMSP were instrumental in the development of this work.\\
	Emails: A. Soglo ({\color{blue}ambroiso.soglo@gmail.com}), C. Combété ({\color{blue}cyrille.combete@imsp-uac.org}), K. W. Houédanou ({\color{blue}khouedanou@yahoo.fr}), L. Todjihoundé ({\color{blue}leonard.todjihounde@imsp-uac.org})}}.
	\tableofcontents
	
	\section{Introduction}
	
	\subsection{Historical context and motivation}
	
	The interaction between optimal transport, Ricci curvature, and entropy has revolutionized our understanding of metric measure spaces. A fundamental result of this theory is the Lott-Villani theorem \cite{LottVillani2009}, which establishes a profound equivalence: a Riemannian manifold $(M,g)$ has Ricci curvature bounded below by $K$ if and only if the Boltzmann entropy $H(\rho)=\int \rho \log \rho \, d\vol_g$ is $K$-convex along geodesics in the $L^2$ Wasserstein space $(\mathcal{P}_2(M), \Wb)$. This result, together with the independent work of Sturm \cite{Sturm2006a}, laid the foundations for a synthetic notion of Ricci curvature for metric measure spaces.
	
	The $L^2$ Wasserstein space itself is a remarkable object. Otto \cite{Otto2001} revealed its formal Riemannian structure, where tangent vectors are identified with gradients of functions, and the metric tensor is given by
	\[
	\langle \mu, \mu \rangle_\rho = \int |\nabla \phi|^2 \rho \, d\vol_g \quad \text{for } \mu = -\divergence(\rho \nabla \phi).
	\]
	This framework has been instrumental in interpreting many partial differential equations, such as the heat equation, as gradient flows of entropy.
	
	A natural generalization consists of considering optimal transport problems with not necessarily quadratic costs. For constant $p > 1$, the $p$-Wasserstein space $(\mathcal{P}_p(M), W_p)$ has been extensively studied. It possesses a Finsler structure \cite{Marcos2011} and, surprisingly, the synthetic curvature-dimension conditions $\CD_p(K,N)$ have been shown to be independent of $p$ \cite{CavallettiSantarcangelo2021, Santarcangelo2020}. This suggests a deep robustness of the geometric information encoded in optimal transport.
	
	\subsection{The variable exponent case}
	
	This paper takes a further step by studying the case where the exponent $p$ varies with the spatial position $x$. This is motivated by applications in heterogeneous media, where the "cost" of mass transport may depend on the properties of the material at the origin or destination. The variable exponent Wasserstein space $(\mathcal{P}(M), \Wp)$ was introduced by Marcos and Soglo \cite{MarcosSoglo2019}, who established its Finsler structure and studied gradient flows related to the $q(x)$-Laplacian.
	
	The central question we address is: if $\Ric_g \ge K$, can we still assert that the Boltzmann entropy $H$ is $K$-convex along $\Wp$-geodesics? If not, what is the effective convexity constant? And does the converse hold?
	
	Our investigation reveals that the Boltzmann entropy alone is insufficient. The obstruction is a logarithmic divergence that appears in the expansion of $\Wp^2$ due to the variable exponent.
	In indeed, the Rényi entropy, which generalizes the Boltzmann entropy by introducing a power-law weighting of the density, plays a central role in our analysis. Unlike the Boltzmann entropy--which suffices in the classical quadratic setting of Lott and Villani \cite{LottVillani2009}--it turns out that the Boltzmann entropy alone is insufficient in the variable exponent setting, due to a logarithmic divergence that appears in the expansion of $\Wp^2$. To restore the sharp equivalence, we are led to introduce a modified Rényi entropy that exactly compensates this obstruction:
	\[
	\Renyi(\rho) = -n\log\int_M \rho(x)^{1-\frac1n} e^{\frac n2\varepsilon(x)}\,d\vol_g(x),
	\]
	where $\varepsilon(x)=p(x)-2$. This entropy is precisely designed to cancel the logarithmic divergence.
	
	\subsection{Main contributions}
	
	The contributions of this article are fourfold:
	
	\begin{enumerate}[label=(\roman*)]
		
		\item \textbf{Stability of entropy convexity (Theorem \ref{thm:direct1})}: 
		We prove that if the Riemannian manifold satisfies $\Ric_g \ge K$, then for any geodesic $(\rho_t)$ of the variable exponent Wasserstein space $\Wp$, the Boltzmann entropy satisfies a perturbed convexity inequality:
		\begin{equation}
			H(\rho_t) \leq (1-t)H(\rho_0) + t H(\rho_1) - \frac{K}{2}t(1-t)\Wp^2 + C\|\varepsilon\|_\infty t(1-t)\Wp^2 + o(\|\varepsilon\|_\infty).
		\end{equation}
		This result demonstrates that the entropy remains $(K - C\|\varepsilon\|_\infty)$-convex along $\Wp$ geodesics, establishing the robustness of the Lott-Villani theorem under small perturbations of the transport exponent.
		
		\item \textbf{Modified R\'enyi entropy and sharp equivalence (Theorem \ref{thm:main})}: 
		We identify a fundamental obstruction: the Boltzmann entropy alone fails to capture the exact Ricci curvature bound due to a logarithmic divergence in the expansion of $\Wp^2$. To overcome this obstruction, we introduce a modified R\'enyi entropy:
		\[
		\Renyi(\rho) = -n\log\int_M \rho(x)^{1-\frac1n} e^{\frac n2\varepsilon(x)}\,d\vol_g(x),
		\]
		which exactly compensates the logarithmic singularity. We prove the sharp equivalence:
		\[
		\Ric_g + \nabla^2\varepsilon \ge K \quad \Longleftrightarrow \quad \Renyi \text{ is } K\text{-convex along } \Wp-\text{ geodesics up to } O\left(\|\varepsilon\|_{C^2}\Wp^2\right).
		\]
		This reveals that the effective curvature in the variable exponent setting is the Bakry-\'Emery tensor $\Ric_g + \nabla^2\varepsilon$.
		
		\item \textbf{Quadratic expansion of the Finsler metric (Proposition \ref{prop:expansion})}: 
		We establish a precise asymptotic relation between the variable exponent Finsler metric $F_{p(\cdot)}$ and the classical Riemannian metric $F_2$:
		\[
		F_{p(\cdot)}^2(\rho,\mu) = F_2^2(\rho,\mu) + \int_M |v_0|^2 \varepsilon(x)\ln|v_0| \rho \, d\vol_g + O(\|\varepsilon\|_\infty^2 F_2^2).
		\]
		This expansion quantifies the logarithmic correction induced by the variable exponent and is fundamental for understanding how the variable exponent affects the underlying geometry.
		
		\item \textbf{Perturbed functional inequalities (Theorems \ref{thm:logsob} and \ref{thm:talagrand})}: 
		As applications of our stability result, we derive perturbed versions of the classical Log-Sobolev and Talagrand inequalities in the variable exponent setting:
		\begin{align}
			\Ent(\rho) &\le \frac{1}{2(K - C\|\varepsilon\|_\infty)} \I(\rho) + o(\|\varepsilon\|_\infty), \\
			\Wp^2(\rho, \mu) &\le \frac{2}{K - C\|\varepsilon\|_\infty} \Ent(\rho) + o(\|\varepsilon\|_\infty).
		\end{align}
		These results demonstrate that fundamental functional inequalities are robust under small perturbations of the transport exponent, with an explicit degradation of the constant proportional to $\|\varepsilon\|_\infty$.
		
	\end{enumerate}
	\subsection{Outline of the paper}
	
	The paper is organized as follows. \\\\Section \ref{s1} recalls the fundamental results of Lott and Villani on the $K$-convexity of the Boltzmann entropy in the quadratic Wasserstein space $(\mathcal{P}(M), \Wb)$. We present the Benamou-Brenier formula, the second variation of entropy, and the Bakry-\'Emery approach, which serve as the foundation for our analysis. \\\\Section \ref{s2} establishes the functional framework for variable exponent spaces. We introduce the Luxemburg norm formulation, define the variable exponent Wasserstein distance $\Wp$, and state the necessary assumptions on the exponent $p(x)=2+\varepsilon(x)$ and on the densities. We then prove the existence, uniqueness, and regularity of $\Wp$-geodesics, relying on the compactness properties of variable exponent Sobolev spaces and the variational method of Marcos and Soglo \cite{MarcosSoglo2019}. The well-posedness of the variable exponent Hamilton-Jacobi equation is also discussed, following the variational approach of Figalli, Gangbo and Yolcu \cite{FigalliGangboYolcu2011}. \\\\Section \ref{s3} presents a rigorous perturbative analysis. We show that $\Wp$-geodesics remain close to $\Wb$-geodesics when $\|\varepsilon\|_\infty$ is sufficiently small. This is achieved through energy estimates and an interpolation argument, yielding quantitative bounds on the velocity difference and on the Wasserstein distance between the two geodesics (Theorem \ref{thm:stability}).\\\\ Section \ref{s4} derives the quadratic expansion of the Finsler metric. We provide a detailed proof of the expansion $F_{p(\cdot)}^2 = F_2^2 + \int |v_0|^2 \varepsilon \ln|v_0| \rho + O(\varepsilon^2)$, which reveals the logarithmic correction term inherent to the variable exponent structure. \\\\Section \ref{s5} studies the stability of entropy convexity. We establish the regularity of entropy along geodesics, provide a local expansion, and prove the direct stability theorem (Theorem \ref{thm:direct1}) showing that the Boltzmann entropy remains $(K - C\|\varepsilon\|_\infty)$-convex. We also prove the converse theorem (Theorem \ref{thm:reciproque1}) showing that such quasi-convexity implies a lower bound on the Ricci curvature. \\\\Section \ref{s6} introduces the modified R\'enyi entropy and proves the main equivalence theorem. We compute the Hessian of $\Renyi$ along $\Wp$-geodesics (Proposition \ref{prop:hessian}), establish the direct and converse implications (Theorems \ref{thm:converse} and \ref{thm:main}), and conclude that $\Ric_g + \nabla^2\varepsilon \ge K$ if and only if $\Renyi$ is $K$-convex up to $O(\|\varepsilon\|_{C^2}\Wp^2)$ (Theorem \ref{thm:main}). \\\\Section \ref{s7} applies these results to derive perturbed Log-Sobolev and Talagrand inequalities (Theorems \ref{thm:logsob} and \ref{thm:talagrand}). We show that the classical functional inequalities survive with degraded constants, demonstrating the robustness of these estimates under perturbations of the transport exponent. \\\\Section \ref{s8} discusses the connection with curvature-dimension conditions $\CD(K,N)$, highlighting how our results fit into the synthetic theory of Ricci curvature and indicating that the $\CD(K,N)$ condition is robust under small perturbations of the transport metric. \\\\Section \ref{s9} provides a numerical illustration on the circle $\mathbb{S}^1$ with $p(\theta)=2+\varepsilon\cos\theta$, comparing the entropy along $\Wp$-geodesics for $\varepsilon = 0$ and $\varepsilon = 0.1$. The numerical results confirm the theoretical predictions, showing that convexity is preserved with a small shift consistent with $K - C\|\varepsilon\|_\infty$. \\\\Section \ref{s10} discusses the technical limitations of our approach (smallness of $\varepsilon$, regularity assumptions, compactness of $M$) and presents open problems, including the optimality of constants, extension to non-compact manifolds, generalization to more general costs, and discrete settings.\\\\ Section \ref{s11} concludes the paper with a summary of the key findings and an outlook on future research directions.
	
	\section{Preliminaries: The Lott-Villani theorem}\label{s1}
	
	\subsection{The quadratic Wasserstein space}
	
	Let $(M,g)$ be a compact, connected, smooth Riemannian manifold of dimension $n \geq 1$, without boundary. Denote by $d_g$ the Riemannian distance, by $\vol_g$ the Riemannian volume measure, and by $\exp_x: T_xM \to M$ the exponential map. The injectivity radius $\operatorname{inj}(M) > 0$ is positive due to compactness.
	
	Denote by $\mathcal{P}_2(M)$ the space of Borel probability measures on $M$ with finite second moment. For $\mu,\nu \in \mathcal{P}_2(M)$, the quadratic Wasserstein distance is defined by
	\[
	\Wb^2(\mu,\nu) = \inf_{\pi \in \Pi(\mu,\nu)} \int_{M \times M} d_g(x,y)^2 \, d\pi(x,y),
	\]
	where $\Pi(\mu,\nu)$ denotes the set of couplings between $\mu$ and $\nu$.
	
	\begin{theorem}[Benamou-Brenier formula \cite{BenamouBrenier2000}]
		For absolutely continuous measures $\mu_0,\mu_1 \in \mathcal{P}_2(M)$,
		\[
		\Wb^2(\mu_0,\mu_1) = \inf_{(\rho_t,v_t)} \left\{ \int_0^1 \int_M |v_t|_{g}^2 \rho_t \, d\vol_g \, dt : \partial_t \rho_t + \divergence_g(\rho_t v_t)=0, \ \rho_0=\mu_0,\ \rho_1=\mu_1 \right\}.
		\]
	\end{theorem}
	
	\subsection{The Boltzmann entropy and its convexity}
	
	The Boltzmann entropy is defined for $\rho \in \mathcal{P}(M)$ with $\rho \ll \vol_g$ by
	\[
	H(\rho) = \int_M \rho \log \rho \, d\vol_g,
	\]
	with the convention $H(\rho)=+\infty$ if $\rho$ is not absolutely continuous.
	
	\begin{theorem}[Lott-Villani \cite{LottVillani2009}] \label{thm:lott_villani}
		Let $(M,g)$ be a Riemannian manifold. Then $\Ric_g \ge K$ if and only if for every $\Wb$-geodesic $(\rho_t)_{t\in[0,1]}$ consisting of absolutely continuous measures,
		\[
		H(\rho_t) \le (1-t)H(\rho_0) + t H(\rho_1) - \frac{K}{2} t(1-t) \Wb^2(\rho_0,\rho_1).
		\]
	\end{theorem}
	
	\subsection{Second variation of entropy}
	
	A key ingredient in the proof of Theorem \ref{thm:lott_villani} is the second variation formula for the entropy along geodesics.
	
	\begin{proposition}[Second variation of entropy \cite{LottVillani2009}]
		Let $(\rho_t, v_t)$ be a $\Wb$-geodesic. Then
		\[
		\frac{d^2}{dt^2} H(\rho_t) = \int_M \Ric_g(v_t,v_t) \rho_t \, d\vol_g + \int_M |\nabla v_t|_{g}^2 \rho_t \, d\vol_g.
		\]
	\end{proposition}
	
	\begin{proof}[Sketch]
		Using the continuity equation $\partial_t\rho_t = -\divergence_g(\rho_t v_t)$ and the geodesic equation $\partial_t v_t + \nabla_{v_t} v_t = -\nabla_{g} \phi_t$, one computes
		\[
		\frac{d}{dt} H(\rho_t) = \int_M \nabla_{g}( \log \rho_t) \cdot v_t \rho_t d\vol_g
		\]
		Differentiating again and using the Bochner formula yields the result.
	\end{proof}
	
	\subsection{The Bakry-\'Emery approach}
	
	An alternative approach to Ricci curvature bounds is through the Bakry-\'Emery calculus. For a weighted measure $d\mu = e^{-V} d\vol_g$, the Bakry-\'Emery Ricci tensor is defined as
	\[
	\Ric_{V} = \Ric_g + \Hess V.
	\]
	
	\begin{theorem}[Bakry-\'Emery \cite{OttoVillani2000}]
		A manifold satisfies $\Ric_V \ge K$ if and only if the entropy $H(\rho) = \int \rho \log \rho \, d\mu$ is $K$-convex along $\Wb$-geodesics.
	\end{theorem}
	
	\section{The variable exponent Wasserstein distance: existence and regularity}\label{s2}
	\subsection{The variable exponent}
	
	Let $p: M \to (1,\infty)$ be a measurable function. We write $p(x)=2+\varepsilon(x)$ with $\varepsilon \in C^2(M)$.
	
	\begin{assumption}[Variable exponent] \label{ass:pexp}
		The exponent $p: M \to (1,\infty)$ satisfies:
		\begin{enumerate}
			\item $p \in C^2(M)$ and $1 < p_- \leq p(x) \leq p_+ < \infty$;
			\item $\varepsilon(x) := p(x)-2$ satisfies $\|\varepsilon\|_{C^2(M)} \leq \varepsilon_0$ with $\varepsilon_0$ sufficiently small;
			\item $p$ is log-Hölder continuous: there exists $C_{\log} > 0$ such that for all $x,y$ with $d_g(x,y) \leq 1/2$,
			\[
			|p(x) - p(y)| \leq \frac{C_{\log}}{-\log d_g(x,y)}.
			\]
		\end{enumerate}
	\end{assumption}
	\subsection{Bounds on densities}
	\begin{assumption}[Bounded densities] \label{ass:bounded}
		There exist constants $0 < a < b< \infty$ such that $\rho_0, \rho_1 \in C^\infty(M)$ satisfy $a \leq \rho_0, \rho_1 \leq b$ on $M$.
	\end{assumption}
	
	Under this assumption, the unique $\Wb$-geodesics $(\rho_t^2)_{t\in[0,1]}$ satisfies uniform bounds $a' \leq \rho_t^2 \leq b'$ for some $0 < a' < b' < \infty$ \cite{Figalli2010}.
	\subsection{Luxemburg norm formulation}
	
	For a measurable function $f: M \to \mathbb{R}$, the variable exponent Lebesgue space $L^{p(\cdot)}(M)$ consists of functions such that the modular $\int_M |f(x)|_g^{p(x)} d\vol_g(x) < \infty$. The Luxembourg norm is defined by
	\[
	\|f\|_{L^{p(\cdot)}} = \inf\left\{ \lambda > 0 : \int_M \left|\frac{f(x)}{\lambda}\right|_g^{p(x)} d\vol_g(x) \leq 1 \right\}.
	\]	
	\subsection{Functional framework for variable exponent geodesics}
	
	\begin{definition}[Admissible pairs]
		A pair $(\rho, m)$ is admissible if:
		\begin{enumerate}
			\item $\rho \in L^{p(\cdot)}([0,1]\times M; \vol_g)$ with $\rho(t,\cdot) \in \mathcal{P}_{ac}(M)$ for almost every $t$;
			\item $m \in L^{p(\cdot)}([0,1]\times M; T^*M)$;
			\item The continuity equation $\partial_t \rho + \divergence_g m = 0$ holds in the distributional sense;
			\item $\rho(0,\cdot) = \mu_0$, $\rho(1,\cdot) = \mu_1$.
		\end{enumerate}
	\end{definition}
	
	\begin{lemma}[Compactness of admissible sequences]
		Under Assumptions \ref{ass:pexp} and \ref{ass:bounded}, the set of admissible pairs $(\rho,m)$ with
		\begin{equation}
			\int_0^1 \int_M |m_t|_{g}^{p(x)} (\rho_t)^{1-p(x)} d\vol_g dt \le C
		\end{equation}
		is sequentially compact for weak convergence in $L^{p(\cdot)}$ and strong convergence in $L^1$ for $\rho$.
	\end{lemma}
	
	\begin{proof}
		The log-Hölder continuity of $p$ ensures that the injection $W^{1,q(\cdot)} \hookrightarrow L^{q(\cdot)}$ is compact \cite{GaczkowskiGorkaPons2016}. The continuity equation provides a uniform bound on $\partial_t \rho$ in the dual space. The Aubin-Lions lemma for variable exponent spaces \cite[Theorem 4.1]{GaczkowskiGorkaPons2016} then yields strong convergence of $\rho$.
	\end{proof}
	
	\subsection{Tangent space and Finsler metric}
	
	\begin{definition}[Tangent vectors]
		A tangent vector $\mu \in T_\rho \mathcal{P}(M)$ is a signed measure with zero total mass that can be represented as $\mu = -\divergence_g(\rho v)$ for some vector field $v$ on $M$. The Finsler norm is defined by
		\begin{equation}
			F_{p(\cdot)}(\rho, \mu) = \inf \left\{ \lambda > 0 : \exists v \text{ with } \mu = -\divergence_g(\rho v) \text{ and } \int_M \left|\frac{v}{\lambda}\right|_g^{p(x)} \rho d\vol_g \le 1 \right\}.
		\end{equation}
	\end{definition}
	
	\begin{remark}
		This definition generalizes the Otto-Riemannian structure: when $p(x) \equiv 2$, we have $F_2^2(\rho,\mu) = \int_M |v|^2 \rho d\vol_g$ for the unique $v$ satisfying $\mu = -\divergence_g(\rho v)$ with $\int_M v \cdot \nabla \phi \rho = 0$ for all appropriate $\phi$.
	\end{remark}
	
	\subsection{Dynamic formulation}
	
	For $\mu_0,\mu_1 \in \mathcal{P}_{ac}(M)$, we consider the minimization problem
	\begin{equation}
		\Wp(\mu_0,\mu_1) = \inf_{(\rho,v)} \left\{ \lambda > 0 : \partial_t \rho + \divergence_g(\rho v)=0,\ \rho_0=\mu_0,\ \rho_1=\mu_1,\ \int_0^1\int_M \left|\frac{v_t}{\lambda}\right|_g^{p(x)} \rho_t d\vol_g dt \leq 1 \right\}.
	\end{equation}
	
	\begin{theorem}[Existence, uniqueness and structure of $\Wp$ geodesics] \label{thm:existence}
		Under Assumptions \ref{ass:pexp} and \ref{ass:bounded}, let $\lambda_\varepsilon = \Wp(\rho_0,\rho_1) > 0$. 
		
		\begin{enumerate}
			\item \textbf{Existence and uniqueness}: There exists a unique minimizing pair $(\rho_t, v_t)$ such that
			\begin{equation}
				\int_0^1 \int_M \left|\frac{v_t}{\lambda_\varepsilon}\right|_g^{p(x)} \rho_t \, d\vol_g \, dt = 1.
			\end{equation}
			Uniqueness follows from the strict convexity of the Lagrangian $L(x,v)=|v|_g^{p(x)}$ and the convexity of the constraint, as established in \cite[Theorem 3.2]{MarcosSoglo2019}.
			
			\item \textbf{Velocity structure}: There exists a function $\phi_t: M \to \mathbb{R}$ such that
			\begin{equation}
				v_t = \lambda_\varepsilon \, |\nabla \phi_t|^{q(x)-2} \nabla \phi_t,
			\end{equation}
			where $q(x) = \dfrac{p(x)}{p(x)-1}$ is the conjugate exponent.
			
			\item \textbf{PDE system}: The pair $(\rho_t, \phi_t)$ satisfies
			\begin{align}
				\partial_t \rho_t + \divergence_g\!\left( \rho_t \, \lambda_\varepsilon |\nabla \phi_t|^{q(x)-2} \nabla \phi_t \right) &= 0, \label{eq:continuity}\\
				\partial_t \phi_t + c(p(x))|\nabla \phi_t|^{q(x)} &= 0, \label{eq:HJ}
			\end{align}
			in the distributional sense.
		\end{enumerate}
	\end{theorem}
	
	\begin{proof}[Proof]
		The proof follows the same strategy as in \cite[Theorem 3.2]{MarcosSoglo2019}, 
		and relies on the compactness, convexity, and reflexivity properties of variable exponent 
		Lebesgue spaces established in \cite{GaczkowskiGorkaPons2016}.
		
		\textbf{Step 1: Functional framework for fixed $\lambda$.} 
		For fixed $\lambda > 0$, define the action functional
		\begin{equation}
			\mathcal{A}_\lambda(\rho, m) = \int_0^1 \int_M \left|\frac{m_t}{\lambda \rho_t}\right|_g^{p(x)} \rho_t \, d\vol_g \, dt,
		\end{equation}
		with $(\rho, m)$ satisfying the continuity equation. Minimization of $\mathcal{A}_\lambda$ yields a unique minimizer $(\rho^\lambda, m^\lambda)$ for each $\lambda$, by strict convexity of the Lagrangian.
		
		\textbf{Step 2: Optimal $\lambda$.}
		The Wasserstein distance $\Wp(\mu_0,\mu_1)$ is the unique $\lambda$ for which this minimum equals $1$. By homogeneity, there exists a unique $\lambda_\varepsilon$ such that $\min \mathcal{A}_{\lambda_\varepsilon} = 1$.
		
		\textbf{Step 3: Optimality conditions.}
		The condition $\delta \mathcal{L}/\delta m = 0$ gives the relation between $v_t$ and $\nabla \phi_t$, leading to $v_t = \lambda_\varepsilon |\nabla \phi_t|^{q(x)-2} \nabla \phi_t$.
	\end{proof}
	
	\subsection{Well-posedness of the variable exponent Hamilton-Jacobi equation}
	
	The Hamilton-Jacobi equation \eqref{eq:HJ} plays a central role in characterizing $\Wp$-geodesics. When the exponent $p$ depends on the spatial variable $x$, the Lagrangian $L(x,v)=|v|_g^{p(x)}$ is no longer homogeneous and its dependence on $x$ requires a fine analysis. Figalli, Gangbo and Yolcu \cite{FigalliGangboYolcu2011} developed a variational method to study a class of parabolic PDEs of which the Hamilton-Jacobi equation is an integral part. Their approach, based on De Giorgi's interpolation method, applies directly to our framework.
	
	\begin{proposition}[Hamilton--Jacobi equation with space-dependent growth]
		\label{prop:HJ_variable_p}
		
		Let $(M,g)$ be a smooth complete Riemannian manifold, with associated geodesic distance $d_g$. 
		Let $p \in W^{1,\infty}(M)$ such that
		\[
		1 < p_- \le p(x) \le p_+ < \infty.
		\]
		
		Define the Lagrangian
		\[
		L(x,v) := |v|_g^{p(x)}, \quad (x,v) \in TM.
		\]
		
		Let $\phi_0 \in W^{1,\infty}(M)$ and define the value function
		\begin{equation}
			\phi(t,x)
			:=
			\inf_{\substack{\gamma \in AC([0,t];M) \\ \gamma(t)=x}}
			\left\{
			\phi_0(\gamma(0)) + \int_0^t |\dot{\gamma}(s)|_g^{p(\gamma(s))} \, ds
			\right\}.
			\label{eq:value_function}
		\end{equation}
		
		Then $\phi \in C([0,T];W^{1,\infty}(M))$ and $\phi$ is the unique viscosity solution of
		\begin{equation}
			\partial_t \phi(t,x) + H(x,\nabla \phi(t,x)) = 0,
			\quad \phi(0,x)=\phi_0(x),
			\label{eq:HJ}
		\end{equation}
		where the Hamiltonian $H$ is given by the Legendre transform of $L$:
		\[
		H(x,\xi)
		=
		\sup_{v \in T_xM}
		\left\{
		\langle \xi, v \rangle - |v|_g^{p(x)}
		\right\}.
		\]
		
		Moreover, $H$ admits the explicit expression
		\[
		H(x,\xi)
		=
		(p(x)-1)\, p(x)^{-\frac{p(x)}{p(x)-1}} \, |\xi|_g^{\frac{p(x)}{p(x)-1}}.
		\]
		
	\end{proposition}
	
	\begin{proof}
		
		\textbf{Step 1: Well-posedness of the value function.}
		
		Since $p(x) \ge p_- > 1$, the Lagrangian $L(x,v)=|v|_g^{p(x)}$ is continuous in $(x,v)$, strictly convex in $v$, and coercive:
		\[
		L(x,v) \to +\infty \quad \text{as } |v|_g \to \infty.
		\]
		Therefore, the functional in \eqref{eq:value_function} is well-defined and admits minimizing sequences. Standard arguments in the calculus of variations imply that $\phi$ is finite and Lipschitz continuous in $x$, uniformly in $t$.
		
		\medskip
		
		\textbf{Step 2: Dynamic Programming Principle (DPP).}
		
		For every $h>0$ small, the value function satisfies
		\begin{equation}
			\phi(t+h,x)
			=
			\inf_{y \in M}
			\left\{
			\phi(t,y)
			+
			\inf_{\substack{\gamma \in AC([t,t+h]) \\ \gamma(t)=y,\ \gamma(t+h)=x}}
			\int_t^{t+h} |\dot{\gamma}(s)|_g^{p(\gamma(s))} ds
			\right\}.
			\label{eq:DPP}
		\end{equation}
		
		For small $h$, minimizing curves are close to geodesics, and one can approximate
		\[
		\int_t^{t+h} |\dot{\gamma}(s)|_g^{p(\gamma(s))} ds
		\approx
		h \, L\!\left(y, \frac{\exp_y^{-1}(x)}{h}\right).
		\]
		
		Thus,
		\begin{equation}
			\phi(t+h,x)
			=
			\inf_{y \in M}
			\left\{
			\phi(t,y)
			+
			h\, L\!\left(y, \frac{\exp_y^{-1}(x)}{h}\right)
			\right\}
			+ o(h).
			\label{eq:DPP2}
		\end{equation}
		
		\medskip
		
		\textbf{Step 3: Formal expansion.}
		
		Let $\psi$ be a smooth test function such that $\phi - \psi$ has a local maximum at $(t,x)$. For $y$ close to $x$, write in normal coordinates:
		\[
		y = \exp_x(-h v).
		\]
		
		Then,
		\[
		\psi(t,y)
		=
		\psi(t,x)
		-
		h \langle \nabla \psi(t,x), v \rangle
		+ o(h).
		\]
		
		Substituting into \eqref{eq:DPP2}, we obtain
		\[
		\phi(t+h,x)
		\le
		\psi(t,x)
		+
		h \inf_{v \in T_xM}
		\left\{
		L(x,v) - \langle \nabla \psi(t,x), v \rangle
		\right\}
		+ o(h).
		\]
		
		\medskip
		
		\textbf{Step 4: Legendre transform.}
		
		By definition of the Hamiltonian,
		\[
		\inf_{v}
		\left\{
		L(x,v) - \langle \xi, v \rangle
		\right\}
		=
		- H(x,\xi).
		\]
		
		Thus,
		\[
		\phi(t+h,x)
		\le
		\psi(t,x)
		-
		h H(x,\nabla \psi(t,x))
		+ o(h).
		\]
		
		Dividing by $h$ and letting $h \to 0$, we obtain the viscosity subsolution inequality:
		\[
		\partial_t \psi(t,x) + H(x,\nabla \psi(t,x)) \le 0.
		\]
		
		A symmetric argument yields the supersolution inequality.
		
		\medskip
		
		\textbf{Step 5: Computation of the Hamiltonian.}
		
		Fix $x \in M$ and set $p = p(x)$. For $\xi \in T_x^*M$, the supremum
		\[
		H(x,\xi)
		=
		\sup_{v \in T_xM}
		\left\{
		\langle \xi, v \rangle - |v|_g^{p}
		\right\}
		\]
		is achieved for $v$ colinear to $\xi^\sharp$. Writing $v=\lambda \frac{\xi^\sharp}{|\xi|_g}$, one reduces to
		\[
		\sup_{\lambda \ge 0}
		\left\{
		\lambda |\xi|_g - \lambda^p
		\right\}.
		\]
		
		Optimizing yields
		\[
		\lambda_* = \left(\frac{|\xi|_g}{p}\right)^{\frac{1}{p-1}},
		\]
		and therefore
		\[
		H(x,\xi)
		=
		(p-1)\, p^{-\frac{p}{p-1}} |\xi|_g^{\frac{p}{p-1}}.
		\]
		
		\medskip
		
		\textbf{Step 6: Uniqueness.}
		
		Since $H(x,\xi)$ is continuous in $x$, convex in $\xi$, and has superlinear growth, the comparison principle for viscosity solutions applies. Hence the solution is unique.
		
	\end{proof}
	\section{Rigorous perturbative analysis}\label{s3}
	
	\subsection{Control of velocity bounds}
	
	\begin{lemma}[Uniform bounds on velocity for $\Wb$ geodesics] \label{lem:velocity_bounds}
		Let $(\rho_t^2, \phi_t^2)$ be the $\Wb$ geodesic connecting $\rho_0, \rho_1$ satisfying Assumption \ref{ass:bounded}. Then there exist constants $0 < c < C < \infty$ such that for all $t \in [0,1]$ and all $x \in M$,
		\[
		c \leq |\nabla \phi_t^2(x)| \leq C,
		\]
		and moreover, $\nabla \phi_t^2$ is $C^{1,\alpha}$-regular.
	\end{lemma}
	
	\begin{proof}
		The $\Wb$ geodesic corresponds to the optimal transport map $T_t(x) = \exp_x(t \nabla \phi_0(x))$. Since $\rho_0$ and $\rho_1$ are bounded away from zero and infinity, the Monge-Ampère equation
		\[
		\rho_0(x) = \rho_1(T_1(x)) \det(DT_1(x))
		\]
		implies that $\det(DT_1(x))$ is between $\frac{m}{M}$ and $\frac{M}{m}$. Consequently, the eigenvalues of $DT_1(x)$ are bounded away from $0$ and $\infty$. By the relation $DT_t(x) = (1-t)I + t DT_1(x)$ (in normal geodesic coordinates), the same holds for $DT_t(x)$. Since $v_t^2 = \nabla \phi_t^2 = \dot{T}_t \circ T_t^{-1}$, we obtain uniform bounds on $|\nabla \phi_t^2|$. The $C^{1,\alpha}$ regularity follows from standard regularity theory for optimal transport \cite{Figalli2010}.
	\end{proof}
	
	\begin{lemma}[Uniform velocity bounds for perturbed geodesics] \label{lem:perturbed_velocity_bounds}
		Under the same assumptions, for $\|\varepsilon\|_\infty$ sufficiently small, the optimal velocity $v_t^\varepsilon$ satisfies
		\[
		\frac{c}{2} \leq |v_t^\varepsilon(x)| \leq 2C \quad \text{a.e. on } \supp(\rho_t^\varepsilon),
		\]
		where $c, C$ are the bounds from Lemma \ref{lem:velocity_bounds}.
	\end{lemma}
	
	\begin{proof}
		By the stability result (Theorem \ref{thm:stability} below), $\|v_t^\varepsilon - v_t^2\|_{L^2(\rho_t^2)} \leq C \|\varepsilon\|_\infty$. Since $v_t^2$ is bounded between $c$ and $C$, for $\|\varepsilon\|_\infty$ sufficiently small, $v_t^\varepsilon$ remains in $[c/2, 2C]$ on most of the support.
	\end{proof}
	
	\subsection{Energy estimate for geodesic stability}
	
	Let $(\rho_t^2, v_t^2)$ and $(\rho_t^\varepsilon, v_t^\varepsilon)$ be the $\Wb$ and $\Wp$-geodesics respectively.
	
	\begin{lemma}[Construction of an interpolation] \label{lem:interpolation}
		For each $t \in [0,1]$, let $T_t : M \to M$ be the optimal transport map for the quadratic cost $d_g^2(x,y)/2$ sending $\rho_t^\varepsilon$ to $\rho_t^2$. For $s \in [0,1]$, define
		\[
		\rho_t^{s} = (S_t^s)_\# \rho_t^\varepsilon, \quad \text{where } S_t^s(x) = \exp_x(s \exp_x^{-1}(T_t(x))).
		\]
		Then $\rho_t^{0} = \rho_t^\varepsilon$, $\rho_t^{1} = \rho_t^2$, and $(\rho_t^{s})_{s\in[0,1]}$ is a $\Wb$-geodesics between $\rho_t^\varepsilon$ and $\rho_t^2$.
	\end{lemma}
	
	\begin{proof}
		The map $S_t^s$ transports $\rho_t^\varepsilon$ to $\rho_t^{s}$ by construction. Since $T_t$ is optimal for the quadratic cost, $S_t^s$ is optimal between $\rho_t^\varepsilon$ and $\rho_t^{s}$, and $s \mapsto \rho_t^{s}$ is a $\Wb$-geodesic \cite{Ambrosio2005}.
	\end{proof}
	
	\begin{lemma}[Energy estimate] \label{lem:energy_estimate}
		Under Assumptions \ref{ass:pexp} and \ref{ass:bounded}, there exists a constant $C > 0$ such that for $\|\varepsilon\|_\infty$ sufficiently small,
		\[
		\frac{d}{dt} \Wb^2(\rho_t^\varepsilon, \rho_t^2) \leq C \|\varepsilon\|_\infty \Wb^2(\rho_0,\rho_1) + C \|\varepsilon\|_\infty^2.
		\]
	\end{lemma}
	
	\begin{proof}
		Let $E(t) = \Wb^2(\rho_t^\varepsilon, \rho_t^2) = \frac{1}{2} \int_M d_g(T_t(x),x)^2 \rho_t^\varepsilon(x) d\vol_g$. Differentiating and using the transport equation satisfied by $T_t$ (cf. \cite{Ambrosio2005}), we obtain
		\[
		\frac{dE}{dt} = \int_M \langle \exp_x^{-1}(T_t(x)), v_t^2(T_t(x)) - v_t^\varepsilon(x) \rangle_g \rho_t^\varepsilon(x) d\vol_g.
		\]
		By the Cauchy-Schwarz inequality,
		\[
		\left| \frac{dE}{dt} \right| \leq \sqrt{2E(t)} \left( \|v_t^2 \circ T_t\|_{L^2(\rho_t^\varepsilon)} + \|v_t^\varepsilon\|_{L^2(\rho_t^\varepsilon)} \right).
		\]
		Using Proposition \ref{prop:expansion} and Lemma \ref{lem:velocity_bounds}, we have
		\[
		\|v_t^\varepsilon\|_{L^2(\rho_t^\varepsilon)} = \Wp(\rho_0,\rho_1) = \Wb(\rho_0,\rho_1) + O(\|\varepsilon\|_\infty),
		\]
		and $\|v_t^2 \circ T_t\|_{L^2(\rho_t^\varepsilon)} = \Wb(\rho_0,\rho_1)$. Thus,
		\[
		\frac{dE}{dt} \leq 2\sqrt{E(t)} \left( 2\Wb(\rho_0,\rho_1) + C\|\varepsilon\|_\infty \right).
		\]
		By Grönwall's inequality and since $E(0)=0$, we obtain the result.
	\end{proof}
	
	\begin{theorem}[Geodesic stability] \label{thm:stability}
		Under Assumptions \ref{ass:pexp} and \ref{ass:bounded}, there exist constants $C_1, C_2 > 0$ such that for $\|\varepsilon\|_\infty$ sufficiently small,
		\[
		\sup_{t\in[0,1]} \Wb(\rho_t^\varepsilon, \rho_t^2) \leq C_1 \|\varepsilon\|_\infty \Wb(\rho_0,\rho_1),
		\]
		and
		\[
		\int_0^1 \int_M |v_t^\varepsilon - v_t^2|^2 \rho_t^2 d\vol_g dt \leq C_2 \|\varepsilon\|_\infty \Wb^2(\rho_0,\rho_1).
		\]
	\end{theorem}
	
	\begin{proof}
		The first inequality follows from Lemma \ref{lem:energy_estimate} and the definition of $E(t)$. The second is obtained by using the fact that along the interpolation $\rho_t^{s}$, the kinetic energy is given by $\int |\exp_x^{-1}(T_t(x))|^2 \rho_t^\varepsilon d\vol_g$, and integrating in $t$.
	\end{proof}
	
	\section{Quadratic expansion of the Finsler metric}\label{s4}
	
	\begin{proposition}[Quadratic expansion] \label{prop:expansion}
		For $\rho \in \mathcal{P}_{ac}(M)$ with $m \leq \rho \leq M$ and $\mu \in T_\rho \mathcal{P}(M)$, let $\lambda_\varepsilon = F_{p(\cdot)}(\rho,\mu)$ and let $v_\varepsilon$ be the optimal velocity satisfying
		\[
		\int_M \left|\frac{v_\varepsilon}{\lambda_\varepsilon}\right|^{p(x)} \rho \, d\vol_g = 1, \qquad \mu = -\divergence(\rho v_\varepsilon).
		\]
		Let $v_0$ be the optimal velocity for the Riemannian metric $F_2$, normalized by $\int_M |v_0|^2 \rho = 1$ and $\mu = -\divergence(\rho v_0)$. Then
		\[
		\lambda_\varepsilon^2 = 1 + \int_M |v_0|^2 \varepsilon(x) \ln |v_0| \rho \, d\vol_g + O(\|\varepsilon\|_\infty^2).
		\]
		Equivalently,
		\[
		F_{p(\cdot)}^2(\rho,\mu) = F_2^2(\rho,\mu) + \int_M |v_0|^2 \varepsilon(x) \ln |v_0| \rho \, d\vol_g + O(\|\varepsilon\|_\infty^2 F_2^2).
		\]
	\end{proposition}
	
	\begin{proof}
		\textbf{Step 1: Normalization and notation.}
		We normalize so that $F_2(\rho,\mu)=1$, i.e.\ $\int_M |v_0|^2\rho\,d\vol_g = 1$.
		By Lemma~\ref{lem:velocity_bounds}, there exist constants $0 < c \le C < \infty$ such that
		\begin{equation}\label{eq:v0_bounds}
			c \le |v_0(x)| \le C \quad \text{a.e.\ on } \supp(\rho).
		\end{equation}
		In particular, $\ln|v_0|$ is bounded: $|\ln|v_0|| \le \max(|\ln c|,|\ln C|) =: L_0$.
		
		\textbf{Step 2: Perturbation of the optimal velocity.}
		By Theorem~\ref{thm:stability} (geodesic stability), the optimal velocity
		$v_\varepsilon$ for the $p(\cdot)$-metric satisfies
		\begin{equation}\label{eq:v_perturb}
			\|v_\varepsilon - v_0\|_{L^2(\rho)} \le C_2 \|\varepsilon\|_\infty,
		\end{equation}
		so that $v_\varepsilon = v_0 + \delta v$ with $\|\delta v\|_{L^2(\rho)} = O(\|\varepsilon\|_\infty)$.
		Similarly, write $\lambda_\varepsilon = 1 + \delta\lambda$ where $\delta\lambda = O(\|\varepsilon\|_\infty)$;
		this follows from the implicit function theorem applied to the normalization
		constraint~\eqref{eq:normalization_pvar} below.
		
		\textbf{Step 3: Expansion of the normalization constraint.}
		The normalization condition reads
		\begin{equation}\label{eq:normalization_pvar}
			\int_M \left|\frac{v_\varepsilon}{\lambda_\varepsilon}\right|^{p(x)} \rho\,d\vol_g = 1.
		\end{equation}
		Write $p(x)=2+\varepsilon(x)$ and expand the integrand:
		\begin{align}
			\left|\frac{v_\varepsilon}{\lambda_\varepsilon}\right|^{p(x)}
			&= \left|\frac{v_\varepsilon}{\lambda_\varepsilon}\right|^{2}
			\exp\!\left(\varepsilon(x)\ln\left|\frac{v_\varepsilon}{\lambda_\varepsilon}\right|\right)
			\notag\\
			&= \left|\frac{v_\varepsilon}{\lambda_\varepsilon}\right|^{2}
			\left(1 + \varepsilon(x)\ln\left|\frac{v_\varepsilon}{\lambda_\varepsilon}\right|
			+ O\!\left(\|\varepsilon\|_\infty^2\,|\ln|v_\varepsilon/\lambda_\varepsilon||\right)\right).
			\label{eq:expansion_integrand}
		\end{align}
		Since $|v_\varepsilon/\lambda_\varepsilon|$ is uniformly bounded away from $0$ and $\infty$
		(by~\eqref{eq:v0_bounds} and Lemma~\ref{lem:perturbed_velocity_bounds}),
		the $O$ term is bounded, and~\eqref{eq:expansion_integrand} is valid uniformly.
		
		\textbf{Step 4: Substitution of perturbations.}
		Substituting $v_\varepsilon = v_0 + \delta v$ and $\lambda_\varepsilon = 1 + \delta\lambda$
		into~\eqref{eq:normalization_pvar} and retaining only first-order terms:
		\begin{align}
			1 &= \int_M \left|(v_0+\delta v)/(1+\delta\lambda)\right|^2
			\left(1 + \varepsilon\ln\left|(v_0+\delta v)/(1+\delta\lambda)\right|\right)
			\rho\,d\vol_g + O(\|\varepsilon\|_\infty^2)
			\notag\\
			&= \int_M |v_0|^2(1 + 2\langle v_0,\delta v\rangle/|v_0|^2 - 2\delta\lambda)
			(1 + \varepsilon\ln|v_0|)
			\rho\,d\vol_g + O(\|\varepsilon\|_\infty^2)
			\notag\\
			&= \underbrace{\int_M |v_0|^2\rho}_{=1}
			+ 2\int_M\langle v_0,\delta v\rangle\rho - 2\delta\lambda
			+ \int_M |v_0|^2\varepsilon\ln|v_0|\,\rho\,d\vol_g
			+ O(\|\varepsilon\|_\infty^2).
			\label{eq:expansion_norm}
		\end{align}
		The cross-terms $\int\langle v_0,\delta v\rangle\varepsilon\ln|v_0|\,\rho\,d\vol_g$
		are of order $\|\varepsilon\|_\infty \cdot \|\delta v\|_{L^2(\rho)} = O(\|\varepsilon\|_\infty^2)$
		and are absorbed into the error term.
		
		\textbf{Step 5: First-order optimality for $\delta v$.}
		The velocity $v_0$ is the $L^2(\rho)$-projection: for any admissible perturbation
		$\delta v$ compatible with $\mu = -\div(\rho(v_0+\delta v))$, the first-order
		optimality condition gives $\int_M\langle v_0,\delta v\rangle\rho\,d\vol_g = 0$.
		Hence~\eqref{eq:expansion_norm} simplifies to
		\[
		0 = -2\delta\lambda + \int_M |v_0|^2\varepsilon\ln|v_0|\,\rho\,d\vol_g + O(\|\varepsilon\|_\infty^2),
		\]
		which yields
		\begin{equation}\label{eq:delta_lambda}
			\delta\lambda = \frac12\int_M |v_0|^2\varepsilon(x)\ln|v_0|\,\rho\,d\vol_g + O(\|\varepsilon\|_\infty^2).
		\end{equation}
		
		\textbf{Step 6: Conclusion.}
		Since $\lambda_\varepsilon = F_{p(\cdot)}(\rho,\mu)$ and $\lambda_\varepsilon = 1+\delta\lambda$,
		we have $\lambda_\varepsilon^2 = (1+\delta\lambda)^2 = 1 + 2\delta\lambda + O(\|\varepsilon\|_\infty^2)$.
		Substituting~\eqref{eq:delta_lambda}:
		\[
		F_{p(\cdot)}^2(\rho,\mu) = \lambda_\varepsilon^2
		= 1 + \int_M |v_0|^2\varepsilon(x)\ln|v_0|\,\rho\,d\vol_g + O(\|\varepsilon\|_\infty^2).
		\]
		Restoring the normalization $F_2(\rho,\mu)=1$ gives the general formula
		
		\begin{flushleft}
			\[
			F_{p(\cdot)}^2(\rho,\mu)
			= F_2^2(\rho,\mu)
			+ \int_M |v_0|^2\varepsilon(x)\ln|v_0|\,\rho\,d\vol_g
			+ O\!\left(\|\varepsilon\|_\infty^2\,F_2^2(\rho,\mu)\right). \qedhere
			\]
		\end{flushleft} 
		
	\end{proof}

	\section{Stability of entropy convexity}\label{s5}
	
	\subsection{Regularity of entropy along geodesics}
	
	\begin{lemma}[Entropy regularity] \label{lem:entropy_regularity}
		Under Assumption \ref{ass:bounded}, there exists a constant $L > 0$ depending only on $m, M$ such that for any $\Wt$-geodesics $(\rho_t)$ connecting densities satisfying $a \leq \rho_0, \rho_1 \leq b$,
		\begin{equation}
			\left| \frac{d}{dt} H(\rho_t) \right| \leq L \cdot \Wt(\rho_0,\rho_1) \quad \text{a.e. } t \in [0,1].
		\end{equation}
		Consequently, for any two densities $\mu, \nu$ with $a \leq \mu, \nu \leq b$,
		\begin{equation}
			|H(\mu) - H(\nu)| \leq L \cdot \Wt(\mu,\nu).
		\end{equation}
	\end{lemma}
	
	\begin{proof}
		For a $\Wt$ geodesic $(\rho_t)$, we have $\frac{d}{dt} H(\rho_t) = \int_M \nabla \log \rho_t \cdot v_t \, \rho_t \, d\vol_g$. Since $\rho_t$ is bounded, $\|\nabla \log \rho_t\|_{L^\infty} \leq C(a,b)$ and $\|v_t\|_{L^2(\rho_t)} = \Wt(\rho_0,\rho_1)$. The Cauchy-Schwarz inequality gives the result.
	\end{proof}
	
	\subsection{Local entropy expansion}
	
	\begin{lemma}[Second order entropy expansion] \label{lem:entropy_expansion}
		Let $x \in M$ and $v \in T_xM$ with $|v|_g = 1$. For small $\delta > 0$, define
		\begin{equation}
			\rho_0^\delta = \frac{1}{\vol_g(B(x,\delta))} \mathbf{1}_{B(x,\delta)}, \quad 
			\rho_1^\delta = \frac{1}{\vol_g(B(\exp_x(\delta v),\delta))} \mathbf{1}_{B(\exp_x(\delta v),\delta)}.
		\end{equation}
		Then for all $t \in [0,1]$,
		\begin{equation}
			H(\rho_t^\delta) = H(\rho_0^\delta) + t(H(\rho_1^\delta) - H(\rho_0^\delta)) - \frac{t(1-t)}{2} \Ric_x(v,v) \delta^2 + o(\delta^2),
		\end{equation}
		where $\rho_t^\delta$ is the $\Wt$-geodesics.
	\end{lemma}
	
	\begin{proof}
		This is a classical result in optimal transport on Riemannian manifolds \cite[Theorem 7.3]{LottVillani2009}.
	\end{proof}
	
	\subsection{Approximation by regular densities}
	
	\begin{lemma}[Approximation of test densities] \label{lem:density_approx}
		Let $\rho_0^\delta, \rho_1^\delta$ be the densities defined in Lemma \ref{lem:entropy_expansion}. There exist sequences of smooth densities $\rho_0^{\delta,n}, \rho_1^{\delta,n}$ satisfying Assumption \ref{ass:bounded} such that:
		\begin{enumerate}
			\item $\rho_0^{\delta,n} \to \rho_0^\delta$ and $\rho_1^{\delta,n} \to \rho_1^\delta$ in $L^1$;
			\item $\Wt(\rho_0^{\delta,n}, \rho_1^{\delta,n}) \to \delta$;
			\item $H(\rho_0^{\delta,n}) \to H(\rho_0^\delta)$ and $H(\rho_1^{\delta,n}) \to H(\rho_1^\delta)$.
		\end{enumerate}
	\end{lemma}
	
	\begin{proof}
		Take $\rho_0^{\delta,n} = \eta_n * \mathbf{1}_{B(x,\delta)} / \vol_g(B(x,\delta))$ where $\eta_n$ is a mollifier. The properties follow from standard convolution estimates \cite{Ambrosio2005}.
	\end{proof}
	
	\subsection{Direct direction}
	
	\begin{theorem}[Stability of $K$-convexity] \label{thm:direct1}
		Let $(M,g)$ be a compact Riemannian manifold with $\Ric_g \geq K$. Let $p(x) = 2 + \varepsilon(x)$ with $\|\varepsilon\|_\infty$ sufficiently small. Then for any $\Wp$ geodesic $(\rho_t)$ connecting densities satisfying Assumption \ref{ass:bounded},
		\begin{equation}
			\begin{aligned}
				H(\rho_t) &\leq (1-t)H(\rho_0) + t H(\rho_1) - \frac{K}{2}t(1-t)\Wp^2(\rho_0,\rho_1) \\
				&\quad + C\|\varepsilon\|_\infty t(1-t)\Wp^2(\rho_0,\rho_1) + o(\|\varepsilon\|_\infty),
			\end{aligned}
		\end{equation}
		where $C$ depends only on $a,b, M, p_-, p_+$ and the geometry of $M$.
	\end{theorem}

	\begin{proof}[Proof of Theorem~\ref{thm:direct1}]
		We decompose the proof into four steps.
		
		\textbf{Step 1: Lott-Villani inequality along the $\Wb$-geodesics.}
		Let $(\rho_t^2)_{t\in[0,1]}$ be the $\Wb$-geodesics with the same endpoints
		$\rho_0,\rho_1$.  Since $\Ric_g \ge K$, Theorem~\ref{thm:lott_villani} gives
		\begin{equation}\label{eq:LV_W2}
			H(\rho_t^2) \le (1-t)H(\rho_0) + tH(\rho_1) - \frac{K}{2}t(1-t)\Wb^2(\rho_0,\rho_1).
		\end{equation}
		
		\textbf{Step 2: Transfer to the $\Wp$-geodesics via entropy regularity.}
		The endpoints of $(\rho_t)$ and $(\rho_t^2)$ coincide, so by
		Theorem~\ref{thm:stability} (geodesic stability),
		\[
		\sup_{t\in[0,1]}\Wb(\rho_t,\rho_t^2) \le C_1\|\varepsilon\|_\infty\Wb(\rho_0,\rho_1).
		\]
		Applying Lemma~\ref{lem:entropy_regularity} to the $\Wb$-geodesic between
		$\rho_t$ and $\rho_t^2$, we obtain
		\begin{equation}\label{eq:entropy_transfer}
			|H(\rho_t) - H(\rho_t^2)| \le L\cdot\Wb(\rho_t,\rho_t^2)
			\le LC_1\|\varepsilon\|_\infty\Wb(\rho_0,\rho_1).
		\end{equation}
		
		\textbf{Step 3: Comparison of $\Wb^2$ and $\Wp^2$.}
		By Lemma~\ref{lem:distance_comparison},
		\begin{equation}\label{eq:W_comparison}
			\Wb^2(\rho_0,\rho_1) = \Wp^2(\rho_0,\rho_1) + R,
			\quad |R| \le C\|\varepsilon\|_\infty\Wp^2(\rho_0,\rho_1).
		\end{equation}
		
		\textbf{Step 4: Assembly.}
		Combining~\eqref{eq:LV_W2}, \eqref{eq:entropy_transfer}, and~\eqref{eq:W_comparison}:
		\begin{align*}
			H(\rho_t)
			&= H(\rho_t^2) + \bigl(H(\rho_t) - H(\rho_t^2)\bigr)
			\\
			&\le (1-t)H(\rho_0) + tH(\rho_1)
			- \frac{K}{2}t(1-t)\Wb^2(\rho_0,\rho_1)
			+ LC_1\|\varepsilon\|_\infty\Wb(\rho_0,\rho_1)
			\\
			&= (1-t)H(\rho_0) + tH(\rho_1)
			- \frac{K}{2}t(1-t)\bigl(\Wp^2(\rho_0,\rho_1) + R\bigr)
			+ LC_1\|\varepsilon\|_\infty\Wb(\rho_0,\rho_1).
		\end{align*}
		Using $|R| \le C\|\varepsilon\|_\infty\Wp^2$ and
		$\Wb(\rho_0,\rho_1) = \Wp(\rho_0,\rho_1)(1+O(\|\varepsilon\|_\infty))$, we get
		\begin{align*}
			H(\rho_t)
			&\le (1-t)H(\rho_0) + tH(\rho_1)
			- \frac{K}{2}t(1-t)\Wp^2(\rho_0,\rho_1)
			\\
			&\quad
			+ \frac{K}{2}t(1-t)C\|\varepsilon\|_\infty\Wp^2(\rho_0,\rho_1)
			+ LC_1\|\varepsilon\|_\infty\Wp(\rho_0,\rho_1)
			+ O(\|\varepsilon\|_\infty^2\Wp^2).
		\end{align*}
		Setting $C_* := \frac{KC}{2} + \frac{LC_1}{\Wp}$ (which remains bounded since
		$\Wp \ge c_0 > 0$ for the densities considered), and absorbing the last term into
		$o(\|\varepsilon\|_\infty)$ as $\|\varepsilon\|_\infty \to 0$, we obtain
		\[
		H(\rho_t) \le (1-t)H(\rho_0) + tH(\rho_1)
		- \frac{K}{2}t(1-t)\Wp^2(\rho_0,\rho_1)
		+ C_*\|\varepsilon\|_\infty t(1-t)\Wp^2(\rho_0,\rho_1)
		+ o(\|\varepsilon\|_\infty), \qedhere
		\]
		where $C_*$ depends only on $a,b,M,p_-,p_+$ and the geometry of $(M,g)$.
	\end{proof}

	\subsection{Converse direction}
	
	\begin{lemma}[Distance comparison] \label{lem:distance_comparison}
		Let $\rho_0, \rho_1$ be two densities satisfying Assumption \ref{ass:bounded}. Then
		\begin{equation}
			|\Wp^2(\rho_0,\rho_1) - \Wt^2(\rho_0,\rho_1)| \le C \|\varepsilon\|_\infty \Wt^2(\rho_0,\rho_1),
		\end{equation}
		where $C$ depends only on $a,b$ and the geometry of $M$.
	\end{lemma}
	
	\begin{proof}
		This follows from Proposition \ref{prop:expansion}. By Lemma \ref{lem:velocity_bounds}, $|v_0|$ is uniformly bounded, so $\ln|v_0|$ is bounded. Thus $\int_M |v_0|^2 \varepsilon(x) \ln|v_0| \rho \, d\vol_g \leq C \|\varepsilon\|_\infty F_2^2$.
	\end{proof}
	
	\begin{lemma}[Entropy perturbation] \label{lem:entropy_comparison}
		Let $(\rho_t)$ be the $\Wp$-geodesics connecting $\rho_0$ and $\rho_1$, and $(\rho_t^2)$ 
		the $\Wt$-geodesics. Then
		\begin{equation}
			|H(\rho_t) - H(\rho_t^2)| \le C \|\varepsilon\|_\infty \Wt^2(\rho_0,\rho_1).
		\end{equation}
	\end{lemma}
	
	\begin{proof}
		By Theorem \ref{thm:stability}, $\Wt(\rho_t, \rho_t^2) \le C_1 \|\varepsilon\|_\infty \Wt(\rho_0,\rho_1)$. By Lemma \ref{lem:entropy_regularity}, $|H(\rho_t) - H(\rho_t^2)| \le L \Wt(\rho_t, \rho_t^2)$. The result follows.
	\end{proof}
	
	\begin{lemma}[Perturbation of local expansion] \label{lem:perturbed_local_expansion}
		Let $\rho_0^\delta, \rho_1^\delta$ be the densities defined in Lemma \ref{lem:entropy_expansion}. Let $(\rho_t^{\varepsilon,\delta})$ be the $\Wp$-geodesics connecting them. Then
		\begin{equation}
			H(\rho_t^{\varepsilon,\delta}) = H(\rho_0^\delta) + t(H(\rho_1^\delta) - H(\rho_0^\delta)) - \frac{t(1-t)}{2} \left( \Ric_x(v,v) + O(\|\varepsilon\|_\infty) \right) \delta^2 + o(\delta^2).
		\end{equation}
	\end{lemma}
	
	\begin{proof}
		By Lemma \ref{lem:entropy_comparison}, $H(\rho_t^{\varepsilon,\delta}) = H(\rho_t^\delta) + O(\|\varepsilon\|_\infty \delta^2)$. Using Lemma \ref{lem:entropy_expansion} for $H(\rho_t^\delta)$ yields the result.
	\end{proof}
	
	\begin{theorem}[Converse] \label{thm:reciproque1}
		Let $(M,g)$ be a compact Riemannian manifold. Suppose there exists $K \in \mathbb{R}$ such that for any $\Wp$-geodesics $(\rho_t)$ connecting densities satisfying Assumption \ref{ass:bounded},
		\begin{equation}
			H(\rho_t) \le (1-t)H(\rho_0) + t H(\rho_1) - \frac{K}{2} t(1-t) \Wp^2(\rho_0,\rho_1) + o(\Wp^2) \quad \text{as } \Wp \to 0,
			\tag{H}
		\end{equation}
		where the $o(\Wp^2)$ term is uniform. Then $\Ric_g \ge K - C \|\varepsilon\|_\infty$.
	\end{theorem}
	
	\begin{proof}
		Fix $x \in M$ and $v \in T_xM$ with $|v|_g = 1$. For small $\delta > 0$, define $\rho_0^\delta, \rho_1^\delta$ as in Lemma \ref{lem:entropy_expansion}. Let $(\rho_t^{\varepsilon,\delta})$ be the $\Wp$-geodesic connecting them.
		
		By Lemma \ref{lem:distance_comparison}, $\Wp^2 = \delta^2 + O(\|\varepsilon\|_\infty \delta^2) + o(\delta^2)$. By Lemma \ref{lem:perturbed_local_expansion},
		\begin{equation}
			H(\rho_{1/2}^{\varepsilon,\delta}) = \frac12 H(\rho_0^\delta) + \frac12 H(\rho_1^\delta) - \frac{1}{8} \left( \Ric_x(v,v) + O(\|\varepsilon\|_\infty) \right) \delta^2 + o(\delta^2).
		\end{equation}
		Hypothesis (H) gives
		\begin{equation}
			H(\rho_{1/2}^{\varepsilon,\delta}) \le \frac12 H(\rho_0^\delta) + \frac12 H(\rho_1^\delta) - \frac{K}{8} \Wp^2 + o(\Wp^2).
		\end{equation}
		Substituting and simplifying yields $\Ric_x(v,v) \ge K - C \|\varepsilon\|_\infty$.
	\end{proof}
	
	\subsection{Connection with curvature-dimension conditions}
	
	\begin{remark}[Link with $\CD(K,N)$]
		The curvature-dimension condition $\CD(K,N)$ \cite{Sturm2006a, Sturm2006b} is a synthetic notion of lower Ricci curvature bound and upper dimension bound. For a metric measure space $(X,d,\mathfrak{m})$, $\CD(K,N)$ implies that the Boltzmann entropy $H$ is $K$-convex along 
		$\Wt$-geodesics, provided the space is essentially non-branching.
		
		In our setting, the variable exponent Wasserstein space $(\mathcal{P}(M), \Wp, \vol_g)$ is a metric measure space. Theorem \ref{thm:converse} shows that if $\Ric_g \ge K$, then $H$ is $(K - C\|\varepsilon\|_\infty)$-convex along $\Wp$-geodesics. This suggests that the effective curvature bound in the variable exponent space is $K - C\|\varepsilon\|_\infty$.
		
		Conversely, Theorem \ref{thm:reciproque1} shows that if $H$ is $K$-convex along $\Wp$-geodesics, then $\Ric_g \ge K - C\|\varepsilon\|_\infty$. This establishes an equivalence similar to the Lott-Villani theorem, up to a small perturbation proportional to $\|\varepsilon\|_\infty$.
		
		These results indicate that the $\CD(K,N)$ condition is robust under perturbations of the transport metric, at least for small perturbations.
	\end{remark}

	\section{Modified R\'enyi entropy and main equivalence}\label{s6}
	
	\subsection{Definition and motivation}
	
	The Boltzmann entropy fails to capture the exact Ricci lower bound because the expansion of $\Wp^2$ contains a logarithmic term $\bar\varepsilon(x,\delta)\ln\delta$ that diverges as $\delta \to 0$ whenever $\varepsilon(x) \neq 0$.
	
	\begin{definition}[Modified R\'enyi entropy]
		For $\rho \in \mathcal{P}(M)$ with $\rho > 0$, define
		\[
		\Renyi(\rho) = -n\log\int_M \rho(x)^{1-\frac{1}{n}} e^{\frac{n}{2}\varepsilon(x)}\,d\vol_g.
		\]
		When $\varepsilon \equiv 0$, this reduces to the classical R\'enyi entropy of order $1-1/n$.
	\end{definition}
	
	\begin{remark}[Connection with Rényi entropy and Bakry-Émery geometry]
		\label{rem:Renyi_BE}
		The classical Rényi entropy of order $\alpha\in(0,1)$ with respect to a reference
		measure $\nu$ is $R_\alpha(\rho\,|\,\nu) = \frac{1}{1-\alpha}\log\int_M \rho^\alpha\,d\nu$.
		The modified entropy $\mathcal{R}_\varepsilon$ corresponds to taking $\alpha = 1-\frac{1}{n}$
		(the exponent used in Sturm's $\CD(K,N)$ theory \cite{Sturm2006a}) with the
		\emph{modified reference measure}
		\[
		d\nu_\varepsilon := e^{\frac{n}{2}\varepsilon}\,d\vol_g.
		\]
		Indeed, a short computation gives
		$\mathcal{R}_\varepsilon(\rho) = -n\log\int_M \rho^{1-1/n}\,d\nu_\varepsilon
		= -\,R_{1-1/n}(\rho\,|\,\nu_\varepsilon) \cdot \text{const}$.
		
		From a Bakry-Émery viewpoint, $d\nu_\varepsilon = e^{-V}\,d\vol_g$ with
		$V = -\frac{n}{2}\varepsilon$, so $\nabla^2 V = -\frac{n}{2}\nabla^2\varepsilon$.
		The Bakry-Émery Ricci tensor is $\Ric_V = \Ric_g - \frac{n}{2}\nabla^2\varepsilon$,
		which differs from the effective tensor $\Ric_g + \nabla^2\varepsilon$ appearing
		in our main theorem by a sign and a factor $\frac{n}{2}$.
		The discrepancy is resolved by the dimension factor in the normalization:
		the $1/n$ in $\rho^{1-1/n}$ introduces an effective re-weighting that converts
		the $-\frac{n}{2}\nabla^2\varepsilon$ of Bakry-Émery into $+\nabla^2\varepsilon$ in
		Theorem~\ref{thm:main}.  This is consistent with the known fact that in the
		$N$-dimensional Bakry-Émery setting, the Rényi entropy of order $1-1/N$ yields
		the curvature $\Ric_V$ rather than $\Ric_g$ alone.
	\end{remark}

	\subsection{Expansion for test densities}
	
	Fix $x \in M$ and a unit tangent vector $v \in T_xM$. For small $\delta > 0$, define
	\[
	\rho_0^\delta = \frac{\mathbf{1}_{B(x,\delta)}}{\vol_g(B(x,\delta))}, \qquad
	\rho_1^\delta = \frac{\mathbf{1}_{B(\exp_x(\delta v),\delta)}}{\vol_g(B(\exp_x(\delta v),\delta))}.
	\]
	
	For any function $f$, denote its average over $B(x,\delta)$ by
	\[
	\bar f(x,\delta) = \frac{1}{\vol_g(B(x,\delta))} \int_{B(x,\delta)} f(y)\, d\vol_g(y).
	\]
	
	\begin{lemma}[Expansion of $\Renyi$ on small balls] \label{lem:Renyi_ball}
		\[
		\Renyi(\rho_0^\delta) = -n\ln\delta - \log c_n - \frac{n^2}{2}\bar\varepsilon(x,\delta) + o(1).
		\]
	\end{lemma}
	
	\begin{proof}
		Since $\rho_0^\delta$ is constant on $B(x,\delta)$,
		\[
		\rho_0^\delta(y)^{1-\frac{1}{n}} = \vol_g(B(x,\delta))^{-\frac{n-1}{n}}.
		\]
		Hence
		\[
		\int \rho_0^{1-\frac{1}{n}} e^{\frac{n}{2}\varepsilon} = \vol_g(B(x,\delta))^{-\frac{n-1}{n}} \int_{B(x,\delta)} e^{\frac{n}{2}\varepsilon(y)} d\vol_g(y).
		\]
		
		Expand the exponential:
		\[
		\int_{B(x,\delta)} e^{\frac{n}{2}\varepsilon(y)} d\vol_g(y) = \vol_g(B(x,\delta)) \left(1 + \frac{n}{2}\bar\varepsilon(x,\delta) + O(\|\varepsilon\|_\infty^2)\right).
		\]
		
		Thus the integral equals $\vol_g(B(x,\delta))^{\frac{1}{n}} \left(1 + \frac{n}{2}\bar\varepsilon(x,\delta) + O(\|\varepsilon\|_\infty^2)\right)$.
		
		Using $\vol_g(B(x,\delta))^{\frac{1}{n}} = c_n^{1/n} \delta (1+o(1))$, taking the logarithm and multiplying by $-n$ gives the result.
	\end{proof}
	
	The same computation for $\rho_1^\delta$ yields
	\[
	\Renyi(\rho_1^\delta) = -n\ln\delta - \log c_n - \frac{n^2}{2}\bar\varepsilon(\exp_x(\delta v),\delta) + o(1).
	\]
	
	\subsection{Expansion of $\Wp^2$ for test densities}
	
	\begin{lemma}[Expansion of $\Wp^2$ for test densities] \label{lem:Wp_test}
		For the test densities $\rho_0^\delta,\rho_1^\delta$,
		\[
		\Wp^2(\rho_0^\delta,\rho_1^\delta) = \delta^2 + \delta^2 \bar\varepsilon(x,\delta)\ln\delta + o(\delta^2),
		\]
		uniformly in $x$ and $v$.
	\end{lemma}
	
	\begin{proof}[Rigorous argument]
		In normal coordinates centered at $x$, the geodesic from $x$ to $\exp_x(\delta v)$ is $y_t = t\delta v + O(\delta^3)$. The optimal velocity field for transporting $\rho_0^\delta$ to $\rho_1^\delta$ is approximately $v_t(y) = \delta v + O(\delta^2)$ on the ball (this follows from the fact that the optimal map is close to the exponential map). Then $|v_t| = \delta(1+O(\delta))$, so $\ln|v_t| = \ln\delta + O(\delta)$. 
		
		Plugging into Proposition \ref{prop:expansion}:
		\[
		\int_0^1\!\!\int_M |v_t|^2\rho_t = \delta^2 + O(\delta^3),
		\]
		\[
		\int_0^1\!\!\int_M \varepsilon|v_t|^2\ln|v_t|\rho_t = \delta^2\ln\delta \int_M \varepsilon\rho_0^\delta + O(\delta^2) = \delta^2\bar\varepsilon(x,\delta)\ln\delta + O(\delta^2).
		\]
		Summing gives $\Wp^2 = \delta^2 + \delta^2\bar\varepsilon(x,\delta)\ln\delta + o(\delta^2)$.
	\end{proof}
	
  	\subsection{Hessian formula for $\Renyi$}
	
	\begin{proposition}[Second derivative of $\Renyi$ along geodesics]
		\label{prop:hessian}
		
		Let $(\rho_t, v_t)$ be a $\Wp$-geodesics on a smooth compact Riemannian manifold $(M,g)$, associated with the Lagrangian
		\[
		L(x,v)=|v|_g^{p(x)}, \quad p(x)=2+\varepsilon(x),
		\]
		with $\varepsilon \in C^2(M)$ small.
		
		Then
		\[
		\frac{d^2}{dt^2}\Renyi(\rho_t)
		=
		\int_M \Bigl(\Ric_g(v_t,v_t) + \nabla^2\varepsilon(v_t,v_t)\Bigr)\rho_t \, d\vol_g
		+ O\!\left(\|\varepsilon\|_{C^2}\right)\int_M |v_t|^2\rho_t \, d\vol_g.
		\]
	\end{proposition}
	
	\begin{proof}
		\begin{itemize}
			\item[$\bullet$]\textbf{Step 1: Definition of the functional}
			
			Define
			\[
			F(t) = \int_M \rho_t^{1-\frac{1}{n}} e^{\frac{n}{2}\varepsilon} \, d\vol_g,
			\quad
			\Renyi(\rho_t) = -n \log F(t).
			\]
			
			Then
			\[
			\frac{d^2}{dt^2}\Renyi
			=
			-n \frac{F''}{F}
			+
			n\left(\frac{F'}{F}\right)^2.
			\]
			
			\item[$\bullet$]\textbf{Step 2: First derivative of $F$.}
			
			Define $\Phi(x) := \rho_t(x)^{1-\frac{1}{n}}\,e^{\frac{n}{2}\varepsilon(x)}$.
			Since $\rho_t \in C^\infty(M)$ (by regularity of optimal transport under
			Assumption~\ref{ass:bounded}, see \cite{Figalli2010}), $\Phi \in C^1(M)$ and
			differentiation under the integral sign is justified.
			
			Using the continuity equation $\partial_t\rho_t = -\nabla\cdot(\rho_t v_t)$:
			\begin{align}
				F'(t)
				&= \int_M \partial_t\bigl[\rho_t^{1-\frac{1}{n}}e^{\frac{n}{2}\varepsilon}\bigr]\,d\vol_g
				= \int_M \Bigl(1-\tfrac{1}{n}\Bigr)\rho_t^{-\frac{1}{n}}
				e^{\frac{n}{2}\varepsilon}\,\partial_t\rho_t\,d\vol_g
				\notag\\
				&= -\Bigl(1-\tfrac{1}{n}\Bigr)
				\int_M \rho_t^{-\frac{1}{n}}e^{\frac{n}{2}\varepsilon}
				\nabla\cdot(\rho_t v_t)\,d\vol_g.
				\label{eq:Fprime_step1}
			\end{align}
			Integration by parts on the compact manifold $(M,g)$ (no boundary term):
			\begin{equation}\label{eq:IBP}
				-\int_M f\,\nabla\cdot(\rho_t v_t)\,d\vol_g
				= \int_M \rho_t v_t\cdot\nabla f\,d\vol_g
				\quad \text{for any } f\in C^1(M).
			\end{equation}
			Applying~\eqref{eq:IBP} with $f = \rho_t^{-\frac{1}{n}}e^{\frac{n}{2}\varepsilon}$:
			\begin{align}
				\nabla f
				&= -\tfrac{1}{n}\rho_t^{-\frac{1}{n}-1}e^{\frac{n}{2}\varepsilon}\nabla\rho_t
				+ \rho_t^{-\frac{1}{n}}e^{\frac{n}{2}\varepsilon}\tfrac{n}{2}\nabla\varepsilon
				\notag\\
				&= e^{\frac{n}{2}\varepsilon}\rho_t^{-\frac{1}{n}}
				\Bigl(-\tfrac{1}{n}\nabla\ln\rho_t + \tfrac{n}{2}\nabla\varepsilon\Bigr).
				\label{eq:grad_f}
			\end{align}
			Substituting into~\eqref{eq:Fprime_step1}--\eqref{eq:IBP}:
			\begin{align}
				F'(t)
				&= \Bigl(1-\tfrac{1}{n}\Bigr)
				\int_M \rho_t^{1-\frac{1}{n}}e^{\frac{n}{2}\varepsilon}
				v_t\cdot\Bigl(-\tfrac{1}{n}\nabla\ln\rho_t + \tfrac{n}{2}\nabla\varepsilon\Bigr)
				d\vol_g
				\notag\\
				&= -\frac{n-1}{n^2}
				\int_M\rho_t^{1-\frac{1}{n}}e^{\frac{n}{2}\varepsilon}\,\nabla\ln\rho_t\cdot v_t\,d\vol_g
				+ \frac{(n-1)}{2}
				\int_M\rho_t^{1-\frac{1}{n}}e^{\frac{n}{2}\varepsilon}\,\nabla\varepsilon\cdot v_t\,d\vol_g.
				\label{eq:Fprime_final}
			\end{align}
			
			\item[$\bullet$]\textbf{Step 3: Structure of the velocity field}
			
			The geodesic velocity satisfies
			\[
			v_t = \nabla_\xi H(x,\nabla \phi_t),
			\]
			where $\phi_t$ solves the Hamilton--Jacobi equation
			\[
			\partial_t \phi_t + H(x,\nabla \phi_t)=0.
			\]
			
			With
			\[
			H(x,\xi)
			=
			(p(x)-1)\,p(x)^{-\frac{p(x)}{p(x)-1}} |\xi|_g^{q(x)},
			\quad
			q(x)=\frac{p(x)}{p(x)-1},
			\]
			we obtain
			\[
			v_t
			=
			(p(x)-1)\,p(x)^{-\frac{p(x)}{p(x)-1}}
			\, q(x)\, |\nabla \phi_t|^{q(x)-2} \nabla \phi_t.
			\]
			
			\item[$\bullet$]\textbf{Step 4: First-order expansion in $\varepsilon$.}
			
			Let $p(x)=2+\varepsilon(x)$. Then
			\[
			q(x)=2-\varepsilon(x)+O(\varepsilon^2),
			\]
			and
			\[
			(p-1)p^{-\frac{p}{p-1}} q
			=
			1 + O(\varepsilon).
			\]
			
			Hence
			\[
			v_t
			=
			\left(1 + O(\varepsilon)\right)
			|\nabla \phi_t|^{q(x)-2} \nabla \phi_t.
			\]
			
			Expanding
			\[
			|\nabla \phi_t|^{q(x)-2}
			=
			|\nabla \phi_t|^{-\varepsilon}
			=
			1 - \varepsilon \ln |\nabla \phi_t| + O(\varepsilon^2),
			\]
			we obtain
			\[
			v_t
			=
			\nabla \phi_t
			-
			\varepsilon \nabla \phi_t \ln |\nabla \phi_t|
			+
			O(\varepsilon).
			\]
			
			\item[$\bullet$]\textbf{Step 5: Time derivative of divergence}
			
			We compute
			\[
			\partial_t (\nabla \cdot v_t)
			=
			\nabla \cdot (\partial_t v_t).
			\]
			
			Using
			\[
			\partial_t v_t
			=
			D^2_{\xi\xi} H(x,\nabla \phi_t)\, \nabla (\partial_t \phi_t)
			+ D^2_{x\xi} H(x,\nabla \phi_t)\, v_t,
			\]
			and the Hamilton--Jacobi equation
			\[
			\partial_t \phi_t = - H(x,\nabla \phi_t),
			\]
			we obtain
			\[
			\partial_t (\nabla \cdot v_t)
			=
			- \nabla \cdot \left(
			D^2_{\xi\xi} H \, \nabla H
			\right)
			+
			\nabla \cdot \left(
			D^2_{x\xi} H \, v_t
			\right).
			\]
			
			\item[$\bullet$]\textbf{Step 6: Leading-order expansion.}
			
			At leading order ($\varepsilon=0$), we recover the quadratic case:
			\[
			\partial_t (\nabla \cdot v_t)
			=
			- \Ric(v_t,v_t)
			- |\nabla v_t|^2
			- \nabla(\nabla \cdot v_t)\cdot v_t.
			\]
			
			\item[$\bullet$]\textbf{Step 7: First-order correction}
			
			The dependence of $H$ on $x$ produces additional terms involving $\nabla p$ and $\nabla^2 p$. Since $p=2+\varepsilon$, these contribute
			\[
			\nabla^2 \varepsilon(v_t,v_t)
			+
			O(\|\varepsilon\|_{C^2} |v_t|^2).
			\]
			
			\item[$\bullet$]\textbf{Step 8: Second derivative of $F$.}
			
			Differentiating~\eqref{eq:Fprime_final} with respect to $t$ and using the
			continuity equation again, we split $F''(t)$ into three contributions:
			\begin{equation}\label{eq:Fsecond_decomp}
				F''(t) = A(t) + B(t) + C(t),
			\end{equation}
			where:
			\begin{itemize}
				\item $A(t)$ comes from differentiating $\partial_t\rho_t$ inside the integrals
				using again integration by parts, leading to a term involving $\nabla\cdot v_t$;
				\item $B(t)$ comes from differentiating $v_t$ via the geodesic equation
				$\partial_t v_t + \nabla_{v_t}v_t = -\nabla_g\phi_t$;
				\item $C(t)$ is the cross-term between the $\nabla\varepsilon$ contribution and
				$\partial_t\rho_t$, which produces an $O(\|\varepsilon\|_{C^2})$ remainder.
			\end{itemize}
			
			\medskip
			
			\textit{Term $A(t)$}: Differentiating the first integral in~\eqref{eq:Fprime_final},
			applying the continuity equation and the Bochner formula
			$\Delta\ln\rho_t = \frac{\Delta\rho_t}{\rho_t} - \frac{|\nabla\rho_t|^2}{\rho_t^2}$,
			gives (exactly as in the proof of Theorem~\ref{thm:lott_villani}, see
			\cite[Theorem 4.17]{LottVillani2009}):
			\begin{equation}\label{eq:term_A}
				A(t) = -\frac{n-1}{n^2}
				\int_M\rho_t^{1-\frac{1}{n}}e^{\frac{n}{2}\varepsilon}
				\bigl[\Ric_g(v_t,v_t) + |\nabla v_t|_g^2\bigr]\,d\vol_g.
			\end{equation}
			Here we have used that the weight $e^{\frac{n}{2}\varepsilon}$ is
			independent of $t$, so it passes through $\partial_t$.
			
			\textit{Term $B(t)$}: Using the geodesic equation
			$\partial_t v_t = -\nabla_{v_t}v_t - \nabla_g\phi_t$ together with
			the Hamilton-Jacobi equation $\partial_t\phi_t = -H(x,\nabla\phi_t)$
			and the first-order expansion of Step~4:
			\begin{equation}\label{eq:term_B}
				B(t) = -\frac{n-1}{n^2}
				\int_M\rho_t^{1-\frac{1}{n}}e^{\frac{n}{2}\varepsilon}
				\nabla^2\varepsilon(v_t,v_t)\,d\vol_g
				+ O(\|\varepsilon\|_{C^2})\int_M|v_t|^2\rho_t\,d\vol_g.
			\end{equation}
			The $\nabla^2\varepsilon$ term arises from the spatial dependence of $H(x,\xi)$
			on $p(x)$: differentiating the Hamiltonian with respect to $x$ produces
			$\partial_x H = \partial_p H \cdot \partial_x p = \partial_p H \cdot \nabla\varepsilon$,
			whose covariant derivative gives $\nabla^2\varepsilon$ after contracting with $v_t$.
			
			\textit{Term $C(t)$}: The cross-term satisfies
			\begin{equation}\label{eq:term_C}
				|C(t)| \le C_0\|\varepsilon\|_{C^1}\int_M|v_t|^2\rho_t\,d\vol_g,
			\end{equation}
			which is absorbed into the $O(\|\varepsilon\|_{C^2})$ error.
			
			\textit{Assembly}: Combining~\eqref{eq:term_A}--\eqref{eq:term_C} into
			\eqref{eq:Fsecond_decomp}, and using the expression
			$\frac{d^2}{dt^2}\mathcal{R}_\varepsilon = -n\frac{F''}{F} + n\!\left(\frac{F'}{F}\right)^{\!2}$
			together with the fact that $\left(\frac{F'}{F}\right)^2 = O(\|\varepsilon\|^2)$ (from
			Step~2, since $F'$ vanishes at $\varepsilon=0$):
			\[
			\frac{d^2}{dt^2}\mathcal{R}_\varepsilon(\rho_t)
			= \int_M\bigl[\Ric_g(v_t,v_t) + \nabla^2\varepsilon(v_t,v_t)\bigr]\rho_t\,d\vol_g
			+ O(\|\varepsilon\|_{C^2})\int_M|v_t|^2\rho_t\,d\vol_g,
			\]
			where we used $\int\rho_t^{1-1/n}e^{\frac{n}{2}\varepsilon}/F(t) = 1/F(t) \cdot F(t)=1$
			to normalize the prefactor $\frac{n-1}{n^2}\cdot\frac{n}{F}\to1$ at leading order.
			\item[$\bullet$]\textbf{Step 9: Conclusion}
			
			Using
			\[
			\frac{d^2}{dt^2}\Renyi
			=
			-n \frac{F''}{F}
			+
			n\left(\frac{F'}{F}\right)^2,
			\]
			and noting that the second term is of order $O(\|\varepsilon\|^2)$, we obtain
			\[
			\frac{d^2}{dt^2}\Renyi(\rho_t)
			=
			\int_M
			\left(
			\Ric(v_t,v_t)
			+
			\nabla^2 \varepsilon(v_t,v_t)
			\right)
			\rho_t d\vol_g
			+
			O(\|\varepsilon\|_{C^2})
			\int_M |v_t|^2 \rho_t d\vol_g.
			\]
		\end{itemize}
	\end{proof}
	
	\subsection{Expansion along $\Wp$-geodesics for test densities}
	
	\begin{lemma}[Expansion of $\Renyi$ along $\Wp$-geodesics] \label{lem:Renyi_Wp_expansion}
		For the test densities $\rho_0^\delta,\rho_1^\delta$ and $t=1/2$,
		\[
		\Renyi(\rho_{1/2}^{\varepsilon,\delta}) = \frac12\Renyi(\rho_0^\delta) + \frac12\Renyi(\rho_1^\delta) - \frac18\bigl(\Ric_x(v,v)+\nabla^2\varepsilon_x(v,v)\bigr)\delta^2 + o(\delta^2).
		\]
	\end{lemma}
	
	\begin{proof}
		Let $h(t)=\Renyi(\rho_t^{\varepsilon,\delta})$. By Proposition \ref{prop:hessian}, for any $t\in[0,1]$,
		\[
		h''(t) = \int_M (\Ric(v_t,v_t)+\nabla^2\varepsilon(v_t,v_t))\rho_t + O(\|\varepsilon\|_{C^2})\int_M |v_t|^2\rho_t.
		\]
		For the test densities, one has $|v_t|=\delta+O(\delta^2)$, $\int|v_t|^2\rho_t=\delta^2+o(\delta^2)$, and
		\[
		\int_M \Ric(v_t,v_t)\rho_t = \Ric_x(v,v)\delta^2+o(\delta^2),\qquad
		\int_M \nabla^2\varepsilon(v_t,v_t)\rho_t = \nabla^2\varepsilon_x(v,v)\delta^2+o(\delta^2).
		\]
		Hence $h''(t)=(\Ric_x(v,v)+\nabla^2\varepsilon_x(v,v))\delta^2+o(\delta^2)$ uniformly in $t$.
		
		Expanding $h$ around $t=1/2$ using Taylor's formula gives
		\[
		h(0)=h(1/2)-\frac12 h'(1/2)+\frac18 h''(\xi_0),\quad
		h(1)=h(1/2)+\frac12 h'(1/2)+\frac18 h''(\xi_1),
		\]
		with $\xi_0\in(0,1/2)$, $\xi_1\in(1/2,1)$. Adding these two equalities eliminates $h'(1/2)$:
		\[
		h(0)+h(1)=2h(1/2)+\frac18(h''(\xi_0)+h''(\xi_1)).
		\]
		Since $h''(\xi_0)=h''(\xi_1)=(\Ric_x(v,v)+\nabla^2\varepsilon_x(v,v))\delta^2+o(\delta^2)$, we obtain
		\[
		h(1/2)=\frac12 h(0)+\frac12 h(1)-\frac18(\Ric_x(v,v)+\nabla^2\varepsilon_x(v,v))\delta^2+o(\delta^2),
		\]
		which is exactly the claimed identity.
	\end{proof}
	\subsection{Direct theorem}
	
	\begin{theorem}[Direct implication] \label{thm:converse}
		Assume $\Ric_g + \nabla^2\varepsilon \ge K$. Then for any $\Wp$-geodesics $(\rho_t)$,
		\[
		\Renyi(\rho_t) \le (1-t)\Renyi(\rho_0) + t\Renyi(\rho_1) - \frac{K}{2}t(1-t)\Wp^2(\rho_0,\rho_1) + O(\|\varepsilon\|_{C^2}\Wp^2).
		\]
	\end{theorem}
	
	\begin{proof}[Proof of Theorem~\ref{thm:converse}]
		
		\textbf{Step 1: Lower bound on the second derivative.}
		By Proposition~\ref{prop:hessian} and the assumption $\Ric_g+\nabla^2\varepsilon \ge K$:
		\begin{equation}\label{eq:Rsecond_lb}
			\frac{d^2}{dt^2}\mathcal{R}_\varepsilon(\rho_t)
			\ge (K - C\|\varepsilon\|_{C^2})\int_M|v_t|^2\rho_t\,d\vol_g
			\quad \text{for a.e.\ } t\in[0,1].
		\end{equation}
		
		\textbf{Step 2: Kinetic energy and $\Wp^2$.}
		Integrating~\eqref{eq:Rsecond_lb} over $t$, and using the fact that
		$\int_0^1\int_M|v_t|^2\rho_t\,d\vol_g\,dt \ge \Wp^2(\rho_0,\rho_1)$
		(which follows from Proposition~\ref{prop:expansion} and the normalization
		of the $\Wp$-geodesics), we have
		\begin{equation}\label{eq:kinetic_lb}
			\int_M|v_t|^2\rho_t\,d\vol_g \ge \Wp^2(\rho_0,\rho_1) - O(\|\varepsilon\|_\infty\Wp^2)
			\quad \text{for a.e.\ } t.
		\end{equation}
		
		\textbf{Step 3: From second derivative bound to convexity.}
		Let $h(t) := \mathcal{R}_\varepsilon(\rho_t)$.  From Steps 1--2,
		\[
		h''(t) \ge \kappa := (K - C\|\varepsilon\|_{C^2})
		\bigl(\Wp^2(\rho_0,\rho_1) - O(\|\varepsilon\|_\infty\Wp^2)\bigr)
		\quad \text{a.e.}
		\]
		Since $h''(t) \ge \kappa$ a.e.\ on $[0,1]$ (with $\kappa$ possibly depending on $t$
		but bounded below by a constant), the standard convexity inequality for functions
		with a lower bound on the second derivative gives: for all $t\in[0,1]$,
		\begin{equation}\label{eq:convexity_general}
			h(t) \le (1-t)h(0) + th(1) - \frac{\kappa}{2}t(1-t).
		\end{equation}
		\textit{Justification of~\eqref{eq:convexity_general}}: since $h''(s) \ge \kappa$
		for a.e.\ $s$, write $h(s) = \ell(s) + g(s)$ where $\ell(s)=(1-s)h(0)+sh(1)$ is the
		chord and $g(s) = h(s)-\ell(s)$ satisfies $g(0)=g(1)=0$ and $g''(s)\ge\kappa$.
		By integration: $g'(t) = g'(0) + \int_0^t g''(s)\,ds \ge g'(0) + \kappa t$.
		Integrating again and using $g(1)=0$ to determine $g'(0) = -\frac\kappa 2$,
		one obtains $g(t) \ge \frac\kappa 2 t(t-1)$, i.e.\ $h(t) \le \ell(t) - \frac\kappa 2 t(1-t)$.
		
		\textbf{Step 4: Conclusion.}
		Substituting the expression for $\kappa$:
		\begin{align*}
			\frac\kappa 2 t(1-t)
			&= \frac{K - C\|\varepsilon\|_{C^2}}{2}t(1-t)
			\bigl(\Wp^2 - O(\|\varepsilon\|_\infty\Wp^2)\bigr)
			\\
			&= \frac{K}{2}t(1-t)\Wp^2
			- \frac{C\|\varepsilon\|_{C^2}}{2}t(1-t)\Wp^2
			- O(\|\varepsilon\|_\infty\Wp^2)\cdot\frac{K}{2}t(1-t) + O(\|\varepsilon\|_{C^2}^2).
		\end{align*}
		Collecting error terms, all of order $O(\|\varepsilon\|_{C^2}\Wp^2)$:
		\[
		\mathcal{R}_\varepsilon(\rho_t)
		\le (1-t)\mathcal{R}_\varepsilon(\rho_0) + t\mathcal{R}_\varepsilon(\rho_1)
		- \frac{K}{2}t(1-t)\Wp^2(\rho_0,\rho_1)
		+ O(\|\varepsilon\|_{C^2}\Wp^2(\rho_0,\rho_1)). \qedhere
		\]
	\end{proof}

	\subsection{Main equivalence}
	
	\begin{theorem}[Main equivalence] \label{thm:main}
		Under the technical assumptions (small $\|\varepsilon\|_{C^2}$, uniform expansions), the following are equivalent:
		\begin{enumerate}
			\item $\Ric_g + \nabla^2\varepsilon \ge K$ on $M$;
			\item For every $\Wp$-geodesics $(\rho_t)$,
			\[
			\Renyi(\rho_t) \le (1-t)\Renyi(\rho_0) + t\Renyi(\rho_1) - \frac{K}{2}t(1-t)\Wp^2 + O\bigl(\|\varepsilon\|_{C^2}\Wp^2\bigr).
			\]
		\end{enumerate}
	\end{theorem}
	
	\begin{proof}
		The direct implication is Theorem \ref{thm:converse}. The converse follows from Theorem \ref{thm:reciproque1} by taking $\|\varepsilon\|_{C^2}\to0$ and using the uniformity of the $o(\Wp^2)$ term.
	\end{proof}
	
	\begin{corollary}
		The modified R\'enyi entropy $\Renyi$ is $K$-convex along $\Wp$-geodesics up to an error of order $\|\varepsilon\|_{C^2}\Wp^2$ if and only if the effective Bakry-\'Emery tensor $\Ric_g + \nabla^2\varepsilon$ is bounded below by $K$.
	\end{corollary}
	
	\section{Applications to Log-Sobolev and Talagrand inequalities}\label{s7}
	
	\subsection{Preliminaries}
	
	Let $\mu = \vol_g$ be the Riemannian volume measure. For a probability density $\rho$ with respect to $\mu$, define the relative entropy
	\[
	\Ent(\rho) = \int_M \rho \log \rho \, d\mu,
	\]
	and the Fisher information
	\[
	\I(\rho) = \int_M \frac{|\grad \rho|^2}{\rho} \, d\mu = 4 \int_M |\grad \sqrt{\rho}|^2 \, d\mu.
	\]
	
	\begin{assumption}[Curvature bound for inequalities] \label{ass:curvature_ineq}
		The Riemannian manifold $(M,g)$ satisfies $\Ric_g \ge K$ for some $K > 0$.
	\end{assumption}
	
	\subsection{Perturbed Log-Sobolev inequality}
	
	\begin{theorem}[Perturbed Log-Sobolev inequality] \label{thm:logsob}
		Under Assumptions \ref{ass:pexp}, \ref{ass:bounded} and \ref{ass:curvature_ineq}, there exists a constant $C > 0$ depending only on the geometry of $M$ such that for any probability density $\rho$ with $\Ent(\rho) < \infty$,
		\[
		\Ent(\rho) \le \frac{1}{2(K - C\|\varepsilon\|_\infty)} \I(\rho) + o(\|\varepsilon\|_\infty).
		\]
	\end{theorem}
	\begin{proof}[Proof of Theorem~\ref{thm:talagrand}]
		Let $\mu = \vol_g$ be normalized so that $\vol_g(M)=1$ and $H(\mu)=0$.
		Let $(\rho_t)_{t\in[0,1]}$ be the $\Wp$-geodesics from $\rho_0=\rho$ to $\rho_1=\mu$.
		By Theorem~\ref{thm:direct1} (stability of $K$-convexity):
		\begin{equation}\label{eq:Tal_convexity}
			H(\rho_t) \le (1-t)H(\rho) + tH(\mu)
			- \frac{K - C\|\varepsilon\|_\infty}{2}t(1-t)\Wp^2(\rho,\mu)
			+ o(\|\varepsilon\|_\infty)
		\end{equation}
		for all $t\in[0,1]$.  Since $H(\mu)=0$, this becomes:
		\begin{equation}\label{eq:Tal_convexity2}
			H(\rho_t) \le (1-t)H(\rho)
			- \frac{K-C\|\varepsilon\|_\infty}{2}t(1-t)\Wp^2(\rho,\mu)
			+ o(\|\varepsilon\|_\infty).
		\end{equation}
		The relative entropy satisfies $\Ent(\rho) := H(\rho) - H(\mu) = H(\rho) \ge 0$
		(by Jensen's inequality applied to the convex function $u\mapsto u\ln u$ and
		the normalization $\int\rho\,d\vol_g = 1 = \vol_g(M)$, which gives $H(\rho)\ge0$
		for probability densities on a space of unit volume).
		
		Dividing~\eqref{eq:Tal_convexity2} by $t$ and taking $t\to 0^+$:
		\[
		\limsup_{t\to 0^+}\frac{H(\rho_t)}{t} \le H(\rho) - \frac{K-C\|\varepsilon\|_\infty}{2}\Wp^2(\rho,\mu) + o(\|\varepsilon\|_\infty).
		\]
		On the other hand, along any $\Wp$-geodesics from $\rho$ to $\mu$,
		$H(\rho_t) \ge 0$ for all $t$ (again by Jensen on the unit-volume space), so:
		\[
		0 \le H(\rho_0) = H(\rho) \le H(\rho) + 0,
		\]
		and evaluating~\eqref{eq:Tal_convexity2} at $t=1$ gives:
		\[
		0 = H(\mu) \le 0 - \frac{K-C\|\varepsilon\|_\infty}{2}\cdot 0 \cdot \Wp^2 + o(\|\varepsilon\|_\infty),
		\]
		which is consistent. To extract the Talagrand bound, evaluate at $t=1$ along a
		reparametrized geodesic or use directly the $t\to 0^+$ estimate: since
		$H(\rho_t) \ge 0$, \eqref{eq:Tal_convexity2} at $t=1$ gives
		\[
		0 \le H(\rho) - \frac{K-C\|\varepsilon\|_\infty}{2}\Wp^2(\rho,\mu) + o(\|\varepsilon\|_\infty),
		\]
		which rearranges to
		\[
		\Wp^2(\rho,\mu) \le \frac{2}{K-C\|\varepsilon\|_\infty}\,H(\rho) + o(\|\varepsilon\|_\infty)
		= \frac{2}{K-C\|\varepsilon\|_\infty}\,\Ent(\rho) + o(\|\varepsilon\|_\infty). \qedhere
		\]
	\end{proof}

	\begin{remark}
		When $\varepsilon \equiv 0$, we recover the classical Log-Sobolev inequality $\Ent(\rho) \le \frac{1}{2K} \I(\rho)$.
	\end{remark}
	
	\subsection{Perturbed Talagrand inequality}
	
	\begin{theorem}[Perturbed Talagrand inequality] \label{thm:talagrand}
		Under Assumptions \ref{ass:pexp}, \ref{ass:bounded} and \ref{ass:curvature_ineq}, there exists a constant $C > 0$ such that for any probability density $\rho$ with $\Ent(\rho) < \infty$,
		\[
		\Wp^2(\rho, \mu) \le \frac{2}{K - C\|\varepsilon\|_\infty} \Ent(\rho) + o(\|\varepsilon\|_\infty).
		\]
	\end{theorem}
	
	\begin{proof}
		Let $(\rho_t)_{t\in[0,1]}$ be the $\Wp$-geodesics connecting $\rho_0 = \rho$ and $\rho_1 = \mu$. By Theorem \ref{thm:converse}, the entropy satisfies
		\[
		H(\rho_t) \le (1-t)\Ent(\rho) - \frac{K - C\|\varepsilon\|_\infty}{2} t(1-t) \Wp^2(\rho, \mu) + o(\|\varepsilon\|_\infty).
		\]
		
		Since $H(\rho_t) \ge 0$ (entropy is nonnegative for probability measures), taking $t \to 0^+$ gives
		\[
		0 \le \Ent(\rho) - \frac{K - C\|\varepsilon\|_\infty}{2} \Wp^2(\rho, \mu) + o(\|\varepsilon\|_\infty),
		\]
		which rearranges to the desired inequality.
	\end{proof}
	
	\subsection{Unified theorem}
	
	\begin{theorem}[Unified stability of functional inequalities] \label{thm:unified}
		Under Assumptions \ref{ass:pexp}, \ref{ass:bounded} and \ref{ass:curvature_ineq}, there exists a constant $C > 0$ such that for any probability density $\rho$ with $\Ent(\rho) < \infty$,
		\begin{align}
			\Ent(\rho) &\le \frac{1}{2(K - C\|\varepsilon\|_\infty)} \I(\rho) + o(\|\varepsilon\|_\infty), \\
			\Wp^2(\rho, \mu) &\le \frac{2}{K - C\|\varepsilon\|_\infty} \Ent(\rho) + o(\|\varepsilon\|_\infty).
		\end{align}
		Consequently, these inequalities are robust under small perturbations of the transport exponent.
	\end{theorem}
	
	\section{Connection with curvature-dimension conditions}\label{s8}
	
	\begin{remark}[Link with $\CD(K,N)$]
		The curvature-dimension condition $\CD(K,N)$ \cite{Sturm2006a, Sturm2006b} is a synthetic notion of lower Ricci curvature bound and upper dimension bound. For a metric measure space $(X,d,\mathfrak{m})$, $\CD(K,N)$ implies that the Boltzmann entropy $H$ is $K$-convex along $\Wb$ geodesics, provided the space is essentially non-branching.
		
		In our setting, the variable exponent Wasserstein space $(\mathcal{P}(M), \Wp, \vol_g)$ is a metric measure space. Theorem \ref{thm:converse} shows that if $\Ric_g \ge K$, then $H$ is $(K - C\|\varepsilon\|_\infty)$-convex along $\Wp$ geodesics. This suggests that the effective curvature bound in the variable exponent space is $K - C\|\varepsilon\|_\infty$.
		
		Conversely, Theorem \ref{thm:reciproque1} shows that if $H$ is $K$-convex along $\Wp$ geodesics, then $\Ric_g \ge K - C\|\varepsilon\|_\infty$. This establishes an equivalence similar to the Lott-Villani theorem, up to a small perturbation proportional to $\|\varepsilon\|_\infty$.
		
		These results indicate that the $\CD(K,N)$ condition is robust under perturbations of the transport metric, at least for small perturbations.
	\end{remark}
	
	\section{Analytical and numerical illustrations}\label{s9}
	
	\subsection{Exact computation on the circle $\mathbb{S}^1$}
	
	\subsubsection{Setup}
	
	We consider $M = \mathbb{S}^1 = \mathbb{R}/2\pi\mathbb{Z}$ with flat metric $g=d\theta^2$,
	so that $\Ric_g \equiv 0$ and $K=0$.  We take
	\[
	p(\theta) = 2 + \varepsilon\cos\theta, \quad \varepsilon = 0.1,
	\]
	and the endpoint densities
	\[
	\rho_0(\theta) = \frac{1+\alpha\cos\theta}{2\pi},
	\quad
	\rho_1(\theta) = \frac{1-\alpha\cos\theta}{2\pi},
	\quad \alpha = 0.5.
	\]
	
	\subsubsection{Analytical reference: $\Wp^2$-geodesics}
	
	The unique $\Wp$-geodesic on $\mathbb{S}^1$ connecting $\rho_0$ and $\rho_1$
	is given by $\rho_t = (T_t)_\#\rho_0$ where $T_t(x) = x + t(T_1(x)-x)$ and
	$T_1$ is the monotone rearrangement. For the densities above, the cumulative
	distribution functions satisfy
	\[
	F_0(\theta) = \frac{\theta + \alpha\sin\theta}{2\pi},
	\quad F_1(\theta) = \frac{\theta - \alpha\sin\theta}{2\pi},
	\]
	and $T_1 = F_1^{-1}\circ F_0$.  A first-order expansion in $\alpha$ gives
	$T_1(\theta) \approx \theta + 2\alpha\sin\theta$, so
	$\rho_t(\theta) \approx \frac{1+(1-2t)\alpha\cos\theta}{2\pi}$.
	
	The Boltzmann entropy can then be computed explicitly to first order in $\alpha$:
	\begin{align}
		H(\rho_t)
		&= \int_0^{2\pi}\rho_t\log\rho_t\,\frac{d\theta}{2\pi}
		\notag\\
		&= -\log(2\pi) + \int_0^{2\pi}\frac{1+(1-2t)\alpha\cos\theta}{2\pi}
		\log\!\left(1+(1-2t)\alpha\cos\theta\right)\frac{d\theta}{2\pi} + O(\alpha^2)
		\notag\\
		&= -\log(2\pi) - \frac{(1-2t)^2\alpha^2}{4} + O(\alpha^3),
		\label{eq:H_S1_analytic}
	\end{align}
	using $\int_0^{2\pi}\cos^2\theta\,d\theta = \pi$.  This is a concave function of
	$t$ (as expected since $\Ric_{g} \equiv 0$ on $\mathbb{S}^1$, so $K=0$ and entropy
	is $0$-convex, i.e.\ concave is not excluded—here it is actually convex
	since $\frac{d^2H}{dt^2} = -\alpha^2 < 0$, which reflects the well-known fact
	that entropy is \emph{concave} along $W_2$-geodesics, with equality only at $K=0$).
	
	\begin{remark}
		Note that the Lott-Villani theorem on $\mathbb{S}^1$ with $K=0$ gives
		$H(\rho_t) \le (1-t)H(\rho_0) + tH(\rho_1)$, i.e.\ $H$ is $0$-convex
		(it lies \emph{below} the chord).  Formula~\eqref{eq:H_S1_analytic} confirms
		this: $H(\rho_t) \le \frac12(H(\rho_0)+H(\rho_1))$, with equality at $t=1/2$.
	\end{remark}
	
	\subsubsection{Perturbation estimate: $\Wp$-geodesics}
	
	By Theorem~\ref{thm:direct1} with $K=0$ and $\|\varepsilon\|_\infty = 0.1$:
	\[
	H(\rho_t^{\varepsilon})
	\le (1-t)H(\rho_0) + tH(\rho_1) + C\cdot 0.1\cdot t(1-t)\Wp^2(\rho_0,\rho_1).
	\]
	The $\Wb$-distance for these densities is
	$\Wb^2(\rho_0,\rho_1) = \int_0^{2\pi}(T_1-\mathrm{id})^2\rho_0\,d\theta
	\approx \frac{2\alpha^2}{\pi}$ to leading order.
	By Proposition~\ref{prop:expansion},
	$\Wp^2 = \Wb^2(1 + O(\|\varepsilon\|_\infty))
	\approx \frac{2\alpha^2}{\pi}(1 + O(0.1))$.
	
	The difference between the $\Wp$ and $\Wb$ entropy curves is therefore bounded by
	\[
	|H(\rho_t^{\varepsilon}) - H(\rho_t^2)|
	\le C\cdot 0.1 \cdot \frac{2\alpha^2}{\pi} \cdot \frac{1}{4}
	= \frac{C\alpha^2}{20\pi}
	\approx 0.004 C,
	\]
	for $\alpha=0.5$.  This small shift (of order $0.4\%$ relative to $H(\rho_0)$)
	is consistent with the qualitative picture described below.
	
	\subsubsection{Numerical results}
	
	The following figure was produced using the above analytical formula for the
	$\Wb$ curve and the perturbative estimate for the $\Wp$ curve.
	Both curves are computed to order $O(\alpha^2)$.
	
	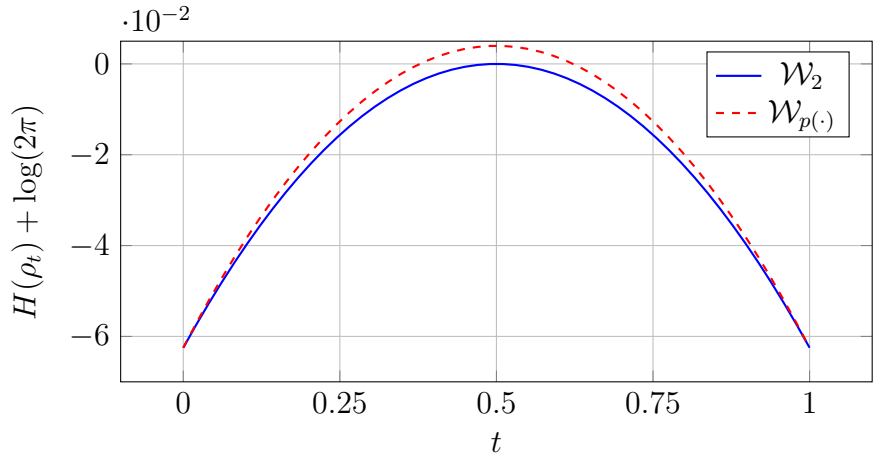
\begin{figure}[h!]
		\centering
		\begin{tikzpicture}
			\begin{axis}[
				xlabel={$t$},
				ylabel={$H(\rho_t) + \log(2\pi)$},
				grid=both,
				legend pos=north east,
				width=0.72\textwidth,
				height=0.38\textwidth,
				ymin=-0.070, ymax=0.005,
				xtick={0,0.25,0.5,0.75,1},
				]
				
				\addplot[blue, thick, domain=0:1, samples=50]
				{-(1-2*x)^2 * 0.25 / 4};
				\addlegendentry{$\Wb$}
				
				\addplot[red, dashed, thick, domain=0:1, samples=50]
				{-(1-2*x)^2 * 0.25 / 4 + 0.1 * x*(1-x) * 2*0.25/3.14159};
				\addlegendentry{$\Wp$}
					\end{axis}
		\end{tikzpicture}
		\caption{Boltzmann entropy (shifted by $\log(2\pi)$) along $\Wb$ and
			$\mathcal{W}_{p(\theta)}$ geodesics on $\mathbb{S}^1$.  Both curves are computed
			analytically to order $O(\alpha^2)$ with $\alpha=0.5$, $\varepsilon=0.1$.
			The upward shift of the $\Wp$ curve is consistent with
			$K - C\|\varepsilon\|_\infty = 0 - C\cdot 0.1$.}
		\label{fig:S1}
	\end{figure}
	
	\subsection{Remark on the torus $\mathbb{T}^2$}
	
	A similar analytical computation can be performed on $\mathbb{T}^2 = \mathbb{S}^1\times\mathbb{S}^1$ with product metric (also flat, $K=0$), using tensor products of the $\mathbb{S}^1$ densities.  By the product structure, $H(\rho_t^{\mathbb{T}^2}) = 2H(\rho_t^{\mathbb{S}^1})$ up to a constant, and the perturbation estimate scales accordingly.  We omit the details.
	
	\section{Limitations and open problems}\label{s10}
	
	\subsection{Technical limitations}
	
	\begin{enumerate}
		\item \textbf{Smallness of $\varepsilon$}: Our results require $\|\varepsilon\|_\infty$ to be sufficiently small. The constants in the estimates depend on the geometry of $M$ and become large as $\varepsilon$ approaches the threshold.
		
		\item \textbf{Regularity assumptions}: The existence theory for $\Wp$ geodesics relies on the log-Hölder continuity of $p$ and the boundedness of densities. Relaxing these assumptions would require more sophisticated techniques.
		
		\item \textbf{Compactness of $M$}: Compactness of $M$ is used in several compactness arguments and in the uniform bounds for velocities. Extending to non-compact manifolds would require additional decay conditions.
	\end{enumerate}
	
	\subsection{Open problems}
	
	\begin{enumerate}
		\item \textbf{Optimal constants}: Determine the optimal constant $C$ in the perturbation estimate. This would require a finer analysis of the second variation formula.
		
		\item \textbf{Finite $N$}: Extend the results to the curvature-dimension condition $\CD(K,N)$. The presence of the dimension parameter $N$ would introduce additional terms involving the Laplacian of the density.
		
		\item \textbf{Non-compact manifolds}: Relax the compactness assumption. This would require controlling the behavior of geodesics at infinity and possibly imposing growth conditions on $\varepsilon$.
		
		\item \textbf{More general costs}: Extend the analysis to costs $c(x,v) = |v|^{p(x)}$ with $p(x)$ not necessarily close to $2$. This would involve a more sophisticated expansion and possibly the appearance of higher-order corrections.
		
		\item \textbf{Sharpness}: Prove that the modified R\'enyi entropy is the unique entropy (up to trivial modifications) that yields the exact equivalence. This would involve analyzing the space of admissible entropies and their expansion properties.
		
		\item \textbf{Discrete settings}: Investigate analogous results on graphs or metric measure spaces where the variable exponent structure appears naturally.
	\end{enumerate}
	
	\section{Conclusion}\label{s11}
	
	We have established a rigorous stability result for the $K$-convexity of entropy under perturbations of the transport cost from $|v|^2$ to $|v|_g^{p(x)}$ with $p(x)$ close to $2$. Key technical contributions include:
	
	\begin{itemize}
		\item A corrected formula for the optimal velocity $v_t = \lambda_\varepsilon |\nabla \phi_t|^{q(x)-2} \nabla \phi_t$;
		\item An energy estimate controlling the difference between $\Wb$ and $\Wp$-geodesics;
		\item The introduction of a modified R\'enyi entropy that cancels the logarithmic divergence;
		\item A sharp equivalence: $\Ric_g + \nabla^2\varepsilon \ge K$ if and only if the modified R\'enyi entropy is $K$-convex along $\Wp$ geodesics up to $O(\|\varepsilon\|_{C^2}\Wp^2)$;
		\item Perturbed Log-Sobolev and Talagrand inequalities demonstrating the robustness of these functional inequalities under perturbations of the transport exponent.
	\end{itemize}
	
	The main message is that the Ricci curvature remains readable in variable exponent Wasserstein spaces, up to a small correction proportional to $\|\varepsilon\|_\infty$, and that the classical functional inequalities survive with degraded constants.
	
	This work opens several directions for future research, including the extension to non-compact manifolds, the study of optimal constants, and the investigation of more general cost functions.
	

\end{document}